\documentclass[11pt,letterpaper]{article}
\usepackage[letterpaper, total={6.5in, 9in}]{geometry}
\usepackage{amsmath}
\usepackage{amsfonts}
\usepackage{color}
\usepackage{latexsym}
\usepackage{tablefootnote}
\usepackage{epsfig}
\usepackage{multirow}
\usepackage{float}
\usepackage{array}
\usepackage{amssymb,amsthm}

\newcommand*{\QEDA}{\hfill\hbox{\vrule width1.0ex height1.0ex}}
\usepackage{algorithm}
\usepackage{algorithmic}
\usepackage{hyperref}

\usepackage{makecell,booktabs}
\usepackage[version=4]{mhchem}
\usepackage[strict]{changepage}

\usepackage{soul}

\usepackage{amsmath}
\usepackage{amsfonts}
\usepackage{latexsym}
\usepackage{amssymb}

\newtheorem{thm}{Theorem}[section]
\newtheorem{theorem}[thm]{Theorem}

\newtheorem{lemma}[thm]{Lemma}

\newtheorem{proposition}[thm]{Proposition}

\newtheorem{definition}[thm]{Definition}

\newtheorem{remark}[thm]{Remark}

\newcommand{\beq}{\begin{equation}}
\newcommand{\eeq}{\end{equation}}
\newcommand{\beqa}{\begin{eqnarray}}
\newcommand{\eeqa}{\end{eqnarray}}
\newcommand{\beqas}{\begin{eqnarray*}}
\newcommand{\eeqas}{\end{eqnarray*}}
\newcommand{\bi}{\begin{itemize}}
\newcommand{\ei}{\end{itemize}}

\newcommand{\vgap}{\vspace{.1in}}
\newcommand{\nn}{\nonumber}

\newcommand{\R}{\mathbb{R}}

\newcommand{\lam}{{\lambda}}

\newcommand{\norm}[1]{\left\Vert#1\right\Vert}

\newcommand{\inner}[2]{\langle #1,#2\rangle}

\newcommand{\argmin}{\mathrm{argmin}\,}

\newcommand{\tx}{\tilde x}

\newcommand{\mE}{\mathbb{E}}
\newcommand{\rd}{\mathrm{d}}

\begin{document}
	\title{Proximal Oracles for Optimization and Sampling}
	\date{April 2, 2024 (first revision: July 8, 2025; second revision: November 11, 2025)}
	\author{
		Jiaming Liang \thanks{Goergen Institute for Data Science and Artificial Intelligence (GIDS-AI) and Department of Computer Science, University of Rochester, Rochester, NY 14620 (email: {\tt jiaming.liang@rochester.edu}). This work was partially supported by GIDS-AI seed funding and AFOSR grant FA9550-25-1-0182.}
		\qquad
		Yongxin Chen \thanks{School of Aerospace Engineering, Georgia Institute of
			Technology, Atlanta, GA, 30332. (email: {\tt yongchen@gatech.edu}). This work was supported by NSF under grants 1942523 and 2008513.}
	}
	\maketitle
	
	\begin{abstract}
 
We consider convex optimization with non-smooth objective function and log-concave sampling with non-smooth potential (negative log density). In particular, we study two specific settings where the convex objective/potential function is either Hölder smooth or in hybrid form as the finite sum of Hölder smooth components. 
To overcome the challenges caused by non-smoothness, our algorithms employ two powerful proximal frameworks in optimization and sampling: the proximal point framework for optimization and the alternating sampling framework (ASF) 
that uses Gibbs sampling on an augmented distribution.
A key component of both optimization and sampling algorithms is the efficient implementation of the proximal map by the regularized cutting-plane method.
We establish its iteration-complexity under both Hölder smoothness and hybrid settings using novel convergence analysis, yielding results that are new to the literature.
We further propose an adaptive proximal bundle method for non-smooth optimization that employs an aggressive adaptive stepsize strategy, which adjusts stepsizes only when necessary and never rejects iterates.
The proposed method is universal since it does not need any problem parameters as input.
Additionally, we provide an exact implementation of a proximal sampling oracle, analogous to the proximal map in optimization, along with simple complexity analyses for both the Hölder smooth and hybrid cases, using a novel technique based on a modified Gaussian integral.
Finally, we combine this proximal sampling oracle and ASF to obtain a Markov chain Monte Carlo method with non-asymptotic complexity bounds for sampling in Hölder smooth and hybrid settings.


	\end{abstract}
	
\section{Introduction}\label{sec:intro}




We are interested in convex optimization problems
\begin{equation}\label{eq:opt}
	\min_{x \in \R^d} f(x)
\end{equation}
as well as log-concave sampling problems
\begin{equation}\label{eq:target}
       {\rm sample} \quad \nu(x) \propto \exp(-f(x)),
\end{equation} 
where $f: \R^d \to \R$ is convex but not necessarily smooth. In sampling, a potential of the distribution $\nu(x)$ is defined as the negative log-density, which is $f(x)$ up to a constant.

Optimization and sampling are two of the most important algorithmic tools at the interface of data science and computation. Optimization has been extensively studied across a wide range of fields, including machine learning, communications, and supply chain management. Over the past two decades, particular attention has been devoted to gradient-based first-order methods. Many classical ideas have been revisited and extended to large-scale optimization, such as the randomized coordinate descent method \cite{nesterov2012efficiency}, the primal–dual hybrid gradient method \cite{chambolle2011first}, and the extragradient method \cite{MR0451121}.
Drawing samples from a given (often unnormalized) probability density plays a crucial role in many scientific and engineering problems that face uncertainty (either physically or algorithmically). Sampling algorithms are widely used in many areas such as statistical inference/estimation, operations research, physics, biology, and machine learning, etc \cite{bertsimas2004solving,durmus2018efficient,dyer1991random,gelman2013bayesian,kalai2006simulated,kannan1997random,krauth2006statistical,sites2003delimiting}. For instance, in Bayesian inference, one draws samples from the posterior distribution to infer its mean, covariance, or other important statistics. Sampling is also heavily used in molecular dynamics to discover new molecular structures. 

This work is along the recent line of research that lies in the interface of sampling and optimization \cite{durmus2019analysis,salim2020primal}. Indeed, sampling is closely related to optimization.
On the one hand, optimization can be viewed as the limiting case of sampling from the distribution $\exp(-f(x)/T)$ as the temperature parameter $T$ (which represents the level of randomness) approaches zero. In this limit, the probability mass increasingly concentrates around the minimizers of $f(x)$.
On the other hand, sampling $\nu(x)$ has an optimization interpretation \cite{jordan1998variational,wibisono2018sampling,yang2020variational}:
the Langevin dynamics in space corresponds to the Fokker-Planck equation, which is the gradient flow
of the relative entropy functional (with respect to $\nu$) in the space of measures with the Wasserstein metric.
The popular gradient-based Markov chain Monte Carlo (MCMC) methods such as Langevin Monte Carlo (LMC) \cite{dalalyan2017theoretical,grenander1994representations,parisi1981correlation,roberts1996exponential}, Metropolis-adjusted Langevin algorithm (MALA) \cite{bou2013nonasymptotic,roberts2002langevin,roberts1996exponential}, and Hamiltonian Monte Carlo (HMC) \cite{neal2011mcmc} resemble the gradient-based algorithms in optimization and can be viewed as the sampling counterparts of them. 


The goal of this paper is to develop efficient proximal algorithms to solve optimization problems \eqref{eq:opt} as well as to draw samples from potentials \eqref{eq:target}, where both $f$ in \eqref{eq:opt} and \eqref{eq:target} lack smoothness (i.e., when $f$ does not have Lipschitz continuous gradient). 
In particular, we consider two settings where the convex objective/potential function $f$ is either Hölder smooth (i.e., the (sub)gradient $f'$ is H\"older-continuous with exponent $\alpha\in [0,1]$) or a hybrid function with multiple Hölder smooth components.
The core of both proximal optimization and sampling algorithms lies in the proximal map of $f$.
We first develop a generic and efficient implementation of this proximal map.
Building on it, we design an adaptive proximal bundle method to solve problem~\eqref{eq:opt}.
Furthermore, by combining the proximal map of $f$ with rejection sampling, we propose a highly efficient approach to realize a proximal sampling oracle, which is used in a proximal sampling framework \cite{lee2021structured,chen2022improved} in the same spirit as the proximal point method for optimization. With those proximal oracles for optimization and sampling in hand, we are finally able to establish the complexity to sample from densities with non-smooth potentials.

We summarize our contributions as follows.
\begin{itemize}
    \item[i)] We analyze the complexity bounds for implementing the proximal map of $f$ using the regularized cutting-plane method in both Hölder smooth and hybrid settings (Section~\ref{sec:opt}).
    The complexity analyses for both Hölder smooth and hybrid cases, presented in Subsections \ref{subsec:bundle} and \ref{subsec:composite-OPT}, respectively, are novel contributions to the literature and employ proof techniques distinct from existing works such as \cite{diaz2023optimal,du2017rate,kiwiel2000efficiency,liang2024unified,liang2020proximal}.
    \item[ii)] We develop an adaptive proximal bundle method (APBM) using the regularized cutting-plane method and a novel adaptive stepsize strategy in the proximal point method, and establish the complexity bound for Hölder smooth optimization (Section~\ref{sec:APBM}). APBM is a universal method as it does not need any problem-dependent parameters as input.
    In contrast to standard universal methods based on conservative line searches on stepsizes, such as the universal primal gradient method of \cite{nesterov2015universal}, APBM has the benefit of adjusting stepsizes only when necessary and never rejects iterates.
    \item[iii)] We propose an efficient scheme to realize the proximal sampling oracle that lacks smoothness and establish novel techniques to bound its complexity.
    Combining the proximal sampling oracle and the proximal sampling framework, we obtain a general proximal sampling algorithm for convex Hölder smooth and hybrid potentials. Finally, we establish complexity bounds for the proximal sampling algorithm in both cases (Section~\ref{sec:sampling}).
    The complexity bounds presented in Section~\ref{sec:sampling} are similar to those in~\cite{fan2023improved}; however, they are derived under the assumption of an exact proximal sampling oracle, whereas~\cite{fan2023improved} considers an inexact implementation of the oracle. The contributions of Section~\ref{sec:sampling} lie in providing much simpler complexity analyses for the exact realization of the proximal sampling oracle in both the Hölder smooth and hybrid cases, compared to the existing analyses in~\cite{liang2022proximal,liang2023a}.
\end{itemize}

It is worth noting that this paper does not aim to establish the optimal complexity of universal methods or to improve the complexity of proximal sampling algorithms. Instead, it develops a regularized cutting-plane method as an efficient implementation of the proximal oracle used in both proximal optimization and sampling, and demonstrates its interesting applications in universal methods and proximal sampling algorithms.

\section{Proximal Optimization and Sampling}

The proximal point framework (PPF), proposed in \cite{martinet1970regularisation} and further developed in \cite{MR0418919,rockafellar1976monotone} (see \cite{ParBoy14} for a modern and comprehensive monograph), is a general class of optimization algorithms that involve solving a sequence of subproblems of the form
\begin{equation}\label{eq:prox-sub}
    x_{k+1} \leftarrow \argmin \left\{f(x)+\frac{1}{2\eta}\|x-x_{k}\|^2: x\in \R^d\right\},
\end{equation}
where $\eta>0$ is a prox stepsize and $\leftarrow$ means the subproblem can be solved either exactly or approximately. When the exact solution is available, we denote 
\[
x_{k+1} = {\rm prox}_{\eta f}(x_k),
\]
where ${\rm prox}_{f}(\cdot)$ is called a proximal map of $f$ and defined as
\begin{equation}\label{def:prox-map}
    {\rm prox}_{f}(y): = \argmin \left\{f(x)+\frac{1}{2}\|x-y\|^2: x\in \R^d\right\}.
\end{equation}
If the subproblem \eqref{eq:prox-sub} does not admit a closed-form solution, it can usually be solved with standard or specialized iterative methods.

Many classical first-order methods in optimization, such as the proximal gradient method, the proximal subgradient method, the primal-dual hybrid gradient method of \cite{chambolle2011first} (also known as the Chambolle-Pock method), the extra gradient method of \cite{MR0451121}
are instances of PPF.
It is worth noting that, by showing that the alternating direction method of multipliers (ADMM) as an instance of PPF, \cite{monteiro2010iteration} gives the first iteration-complexity result of ADMM for solving a class of linearly constrained convex programming problems.

Another example of PPF is the proximal bundle method, which was first proposed in \cite{lemarechal1975extension,lemarechal1978nonsmooth,mifflin1982modification,wolfe1975method} and further developed in \cite{diaz2023optimal,du2017rate,frangioni2002generalized,kiwiel2000efficiency,liang2024unified,liang2020proximal,de2014convex,ruszczynski2011nonlinear,van2017probabilistic}. Notably, inspired by the PPF viewpoint, papers \cite{liang2024unified,liang2020proximal} develop a variant of the proximal bundle method and establish the optimal iteration-complexity, which is the first optimal complexity result for proximal bundle methods.
Recent works \cite{kong2019complexity,jliang2018double,kong2021accelerated,liang2023proximal} have also applied PPF to solve weakly convex optimization and weakly convex-concave min-max problems.

{\bf Proximal map in sampling.} Sampling shares many similarities with optimization. An interesting connection between the two problems is through the algorithm design and analysis from the perspective of PPF.
The alternating sampling framework (ASF) introduced in \cite{lee2021structured} is a generic framework for sampling from a distribution $\pi^X(x) \propto \exp(-f(x))$.
Analogous to PPF in optimization, ASF with stepsize $\eta>0$ repeats the two steps as in Algorithm~\ref{alg:ASF}.
\begin{algorithm}[H]
	\caption{Alternating Sampling Framework \cite{lee2021structured}}
	\label{alg:ASF}
	\begin{algorithmic}
		\STATE 1. Sample $y_{k}\sim \pi^{Y|X}(y\mid x_k) \propto \exp\left(-\frac{1}{2\eta}\|x_k-y\|^2\right)$
		\STATE 2. Sample $x_{k+1}\sim \pi^{X|Y}(x \mid y_{k}) \propto \exp\left(-f(x)-\frac{1}{2\eta}\|x-y_{k}\|^2\right)$
	\end{algorithmic}
\end{algorithm}
ASF is a special case of Gibbs sampling \cite{GemGem84} of the joint distribution 
    \[
        \pi(x,y) \propto \exp\left(-f(x)-\frac{1}{2\eta}\|x-y\|^2\right).
    \]
Starting from the original paper~\cite{lee2021structured} that proposes ASF, subsequent works have refined and extended this framework. In particular, \cite{chen2022improved} provides an improved theoretical analysis of ASF, and \cite{yuan2023class} studies Gibbs sampling based on ASF for structured log-concave distributions over networks.
In Algorithm~\ref{alg:ASF}, sampling $y_k$ given $x_k$ in step 1 can be easily done since $\pi^{Y|X}(y\mid x_k) = {\cal N}(x_k,\eta I)$ is a simple Gaussian distribution.
Sampling $x_{k+1}$ given $y_k$ in step 2 is however a nontrivial task; it corresponds to the so-called restricted Gaussian oracle (RGO) for $f$ introduced in \cite{lee2021structured}, which is defined as follows.

\begin{definition}
Given a point $y\in \R^d$ and stepsize $\eta >0$, the RGO for $f:\R^d\to \R$ is a sampling oracle that returns a random sample from a distribution proportional to $\exp(-f(\cdot) - \|\cdot-y\|^2/(2\eta))$.
\end{definition}

RGO is an analog of the proximal map \eqref{def:prox-map} in optimization.
To use ASF in practice, one needs to efficiently implement RGO. Some examples of $f$ that admit a computationally efficient RGO have been presented in \cite{mou2019efficient,shen2020composite}. These instances of $f$ have simple structures such as coordinate-separable regularizers, $\ell_1$-norm, and group Lasso. To apply ASF on a general potential function $f$, developing an efficient implementation of the RGO is essential.


A rejection sampling-based implementation of RGO for general convex nonsmooth potential function $f$ with bounded Lipschitz constant is given in \cite{liang2022proximal}. 
If the stepsize $\eta$ is small enough, then it only takes a constant number of rejection steps to generate a sample according to RGO in expectation.
Another exact realization of RGO is provided in \cite{liang2023a} for nonconvex hybrid potential $f$ satisfying Hölder continuous conditions. It is also shown that the expected number of rejections to implement RGO is a small constant if $\eta$ is small enough. Other inexact realizations of RGO based on approximate rejection sampling are studied in \cite{gopi2022private,fan2023improved}. 
See Table \ref{tab:rgo} for a clear comparison.
In all these implementations, a key step is realizing the proximal map \eqref{def:prox-map}. It is worth noting that \cite{liang2023a} also connects ASF with other well-known Langevin-type sampling algorithms such as Langevin Monte Carlo (LMC) and Proximal Langevin Monte Carlo (PLMC) via RGO. In a nutshell, \cite{liang2023a} shows that both LMC and PLC are instances of ASF but with approximate implementations of RGO, which always accept the sample from the proposal distribution without rejection. Hence, this provides an alternative interpretation of why the samples generated by LMC are biased, while those produced by ASF are unbiased.

\begin{table}[t]
    \centering
    \begin{tabular}{ccc}
        \toprule
        \textbf{Papers} & \textbf{RGO implementation} & \textbf{Stepsize $\eta$} \\
        \midrule
        \cite{liang2022proximal,liang2023a} & Exact & Small \\ 
        \cite{gopi2022private,fan2023improved} & Approximate & Large \\ 
        \bottomrule
    \end{tabular}
    \caption{Comparison of different RGO implementations and corresponding stepsizes.}
    \label{tab:rgo}
\end{table}

Based on the cutting-plane method, this paper develops a generic and efficient implementation of the proximal map \eqref{def:prox-map} and applies the proximal map in both optimization and sampling. For optimization, we use this proximal map and an adaptive stepsize rule to design a universal bundle method. For sampling, we combine this proximal map and rejection sampling to realize the RGO, and then propose a practical and efficient proximal sampling algorithm based on it.

For both optimization and sampling, we consider two specific scenarios: 1) $f$ is Hölder smooth, i.e., $f$ satisfies 
\begin{equation}\label{ineq:semi-smooth}
    \|f'(u) - f'(v)\| \le L_\alpha\|u-v\|^\alpha,\quad \forall u, v\in \R^d,
    \end{equation}
where $f'$ denotes a subgradient of $f$, $\alpha \in [0,1]$,  and $L_\alpha>0$; and 2) $f$ is a hybrid function of Hölder smooth components, i.e., $f$ satisfies 
    \begin{equation}\label{ineq:composite}
    \|f'(u) - f'(v)\| \le \sum_{i=1}^n L_{\alpha_i}\|u-v\|^{\alpha_i},\quad \forall u, v\in \R^d,
    \end{equation}
where $\alpha_i \in [0,1]$ and $L_{\alpha_i}>0$ for every $1\le i \le n$.
When $\alpha=0$, \eqref{ineq:semi-smooth} reduces to a Lipschitz continuous condition, and when $\alpha=1$, it reduces to a smoothness condition.
It follows from \eqref{ineq:semi-smooth} and \eqref{ineq:composite} that for every $u,v \in \R^d$,
\begin{equation}\label{ineq:semi}
    f(u) - f(v) - \inner{f'(v)}{u-v} \le \frac{L_\alpha}{\alpha+1} \|u-v\|^{\alpha+1},
\end{equation}
and
\begin{equation}\label{ineq:hybrid}
    f(u) - f(v) - \inner{f'(v)}{u-v} \le \sum_{i=1}^n \frac{L_{\alpha_i}}{\alpha_i+1} \|u-v\|^{\alpha_i+1}.
\end{equation}
The proof is given in Appendix~\ref{sec:technical}.

\paragraph{Example.}
Consider the $\ell_p$ regression problem with data $\{(a_i,b_i)\}_{i=1}^n$ where $a_i\in \R^d$ and $b_i \in \R$ for $i=1,\ldots,n$,
\begin{equation}\label{eq:lp}
    f(x) = \frac{1}{n}\sum_{i=1}^n |a_i^\top x - b_i|^p,
\qquad 1 \le p \le 2.
\end{equation}
Define $\phi(t) = |t|^p$, then $\phi'(t) = p\,\operatorname{sign}(t)\,|t|^{p-1}$ and
\[
f'(x) = \frac{1}{n}\sum_{i=1}^n \phi'(a_i^\top x - b_i)\,a_i.
\]
It is shown in Lemma~\ref{lem:Lip} of Appendix~\ref{sec:technical} that $\phi'$ is Hölder continuous with exponent $p-1$ and constant $p 2^{2-p}$.
For any $x,y \in \R^d$, let $u_i = a_i^\top x - b_i$ and $v_i = a_i^\top y - b_i$, using the Hölder continuity of $\phi'$,  we derive
\begin{align*}
&\|f'(x) - f'(y)\|
= \Big\| \frac{1}{n}\sum_{i=1}^n 
  \big(\phi'(u_i) - \phi'(v_i)\big)a_i \Big\| \\
\le &\frac{1}{n}\sum_{i=1}^n
   |\phi'(u_i) - \phi'(v_i)|\,\|a_i\| 
\le \frac{p\,2^{2-p}}{n}
   \Big(\sum_{i=1}^n \|a_i\|^p\Big)
   \|x - y\|^{p-1}.
\end{align*}
Hence, $f$ satisfies the Hölder smoothness condition \eqref{ineq:semi-smooth} with
\[
\alpha = p-1,\qquad
L_\alpha
= \frac{p\,2^{2-p}}{n}\sum_{i=1}^n \|a_i\|^{p}.
\]
The $\ell_p$ regression can be extended to mixed-exponent regression as an example of the hybrid case \eqref{ineq:composite}, where 
\begin{equation}\label{eq:mixed}
    f(x) = \frac{1}{n}\sum_{i=1}^{n} |a_i^\top x - b_i|^{p_i}, 
\qquad 1 \le p_i \le 2,
\end{equation}
and
\[
\alpha_i = p_i-1,\qquad
L_{\alpha_i} = \frac{p_i\,2^{2-p_i}}{n} \|a_i\|^{p_i}.
\]


The above objective functions $f$ in \eqref{eq:lp} and \eqref{eq:mixed} can also appear as the potential energy in 
Bayesian inference. 
Instead of minimizing $f(x)$ to obtain a point estimate (e.g., the maximum a posteriori or MAP solution), one may consider sampling $\nu(x) \propto \exp(-f(x))$ for quantifying uncertainty around the MAP solution.


Throughout the analysis in this paper, we use the following notation. 
When presenting complexity results, ${\cal O}(\cdot)$ denotes the standard ``big-O'' notation, 
while $\tilde {\cal O}(\cdot)$ suppresses polylogarithmic factors. 
We also write $a \asymp b$ to indicate that $a$ and $b$ are of the same order, 
i.e., there exist positive constants $c_1, c_2 > 0$ such that 
$c_1 a \le b \le c_2 a$.

\section{Algorithm and Complexities for the Proximal Subproblem}\label{sec:opt}

The proximal subproblem \eqref{eq:prox-sub} generally does not admit a closed-form solution. We design an iterative method that approximately solves \eqref{eq:prox-sub} and derive the corresponding iteration-complexities for Hölder smooth and hybrid $f$ in Subsections \ref{subsec:bundle} and \ref{subsec:composite-OPT}, respectively.

Given a point $y\in \R^d$, we consider the optimization problem
\begin{equation}\label{eq:subproblem}
    f_y^\eta (x^*)= \min \left\{ f_y^\eta(x)= f(x) + \frac{1}{2\eta}\|x-y\|^2: x\in \R^d\right \}
\end{equation}
and aim at obtaining a $\delta$-solution, i.e., a point $\bar x$ such that $f_y^\eta(\bar x) - f_y^\eta(x^*) \le \delta$.
In both Hölder smooth and hybrid settings, we use a regularized cutting-plane method (Algorithm \ref{alg:PBS}), which is usually used in the proximal bundle method \cite{liang2020proximal,liang2024unified} for solving convex non-smooth optimization problems. 
We remark that though Algorithm \ref{alg:PBS} is widely used in the proximal bundle method and is not new, the complexity analyses (i.e., Theorems \ref{thm:bundle} and \ref{thm:bundle-hybrid}) for Hölder smooth and hybrid functions $f$ are lacking.

Since the prox center $y$ is fixed throughout this section, we simplify the notation by writing $f_y^\eta$ as $f^\eta$ in this section to ease readability.

\begin{algorithm}[H]
	\caption{Regularized Cutting-plane Method} 
	\label{alg:PBS}
	\begin{algorithmic}
        \REQUIRE Let $y\in \R^d$, $\eta>0$, and $\delta>0$ be given, and set $x_0=\tx_0=y$, $j=1$, and $f_0^\eta(x_0)=-\infty$.
        \WHILE{$f^\eta(\tx_{j-1}) - f_{j-1}^\eta(x_{j-1}) > \delta$}
        \STATE 
        \vspace{-5mm}
        \begin{align}
            f_j(x) &= \max \left\lbrace  f(x_i)+\inner{f'(x_i)}{x-x_i} : \, 0\le i\le j-1\right\rbrace, \label{eq:fj} \\
            x_j &=\argmin
	        \left\lbrace f_j^\eta(x):= f_j(x) +\frac{1}{2\eta}\|x- y\|^2 : x\in  \R^d \right\rbrace, \label{def:xj} \\
            \tx_j &= \argmin\left\lbrace f^\eta(x): x\in \{x_j, \tx_{j-1}\}\right\rbrace, \label{def:tx} \\
            j &\gets j + 1. \nn
        \end{align}
        \vspace{-5mm}
        \ENDWHILE
        \RETURN $J=j-1$, $x_J$, and $\tilde x_J$.
	\end{algorithmic}
\end{algorithm}

The basic idea of Algorithm \ref{alg:PBS} is to approximate $f$ with piece-wise affine functions constructed by a collection of cutting-planes and solve the resulting simplified problem \eqref{def:xj}. As the approximation becomes more and more accurate, the best approximate solution $\tx_j$ converges to the solution $x^*$ to \eqref{eq:subproblem}.
Subproblem \eqref{def:xj} can be reformulated into convex quadratic programming with $j$ affine constraints and hence is solvable.

The following technical lemma summarizes basic properties of Algorithm~\ref{alg:PBS}. It is useful in the complexity analysis for both optimization and sampling.

\begin{lemma}\label{lem:bundle}
Assume $f$ is convex. 
For every $j\ge 1$, define 
\begin{equation}\label{def:tj}
    \delta_j:=f^\eta(\tx_j) - f_j^\eta(x_j).
\end{equation}
Let $J, x_J, \tilde x_J$ be the outputs of Algorithm \ref{alg:PBS}, then the following statements hold:
\begin{itemize}
    \item[a)] $\{f_j\}$ serves as a sequence of non-decreasing lower approximations of $f$: $f_j(x)\le f_{j+1}(x)$ and $f_j(x)\le f(x)$, $\forall x\in \R^d$ and $\forall j\ge 1$;

    \vspace{0.2cm}
    
    \item[b)] direct consequence of \eqref{def:xj}: $f_j^\eta(x_j) + \|x-x_j\|^2/(2\eta) \le f_j^\eta(x)$, $\forall x\in \R^d$ and $\forall j \ge 1$;

    \vspace{0.2cm}
    
    \item[c)] $\{\delta_j\}$ is a decreasing sequence:
    $\delta_J \le \delta$ and $\delta_{j+1} + \frac{1}{2\eta}\|x_{j+1}-x_j\|^2 \le \delta_j$, $\forall j \ge 1$;

    \vspace{0.2cm}
    
    \item[d)] solution guarantee for $x_J$ and $\tx_J$:
    $f^\eta(\tilde x_J) - f^\eta(x) \le\delta - \frac1{2\eta}\|x_J - x\|^2$, $\forall x\in \R^d$;

    \vspace{0.2cm}
    
    \item[e)] optimality condition of \eqref{eq:subproblem}: $ - \frac{1}{\eta} (x^*-y) \in\partial f(x^*)$ where $\partial f$ denotes the subdifferential of $f$.
\end{itemize}
\end{lemma}

\begin{proof}
a) The first inequality follows from the definition of $f_j$ in step 2 of Algorithm \ref{alg:PBS}.
The second inequality directly follows from the definition of $f_j$ and the convexity of $f$.

b) Noting that $f_j^\eta$ as the objective function of \eqref{def:xj} is $(1/\eta)$-strongly convex, it thus follows from Theorem 5.25 of \cite{beck2017first} that 
\[
f_j^\eta(x) - f_j^\eta(x_j) \ge \frac{1}{2\eta}\|x-x_j\|^2, \quad \forall x\in \R^d.
\]
Hence, this statement follows.

c) This first inequality immediately follows from \eqref{def:tj} and step 4 of Algorithm \ref{alg:PBS}.

Using the first inequality in \ref{lem:bundle}(a) and \ref{lem:bundle}(b) with $x=x_{j+1}$, we obtain
\[
f_{j+1}^\eta(x_{j+1}) \ge f_j^\eta(x_{j+1}) \ge f_j^\eta(x_j) +\frac{1}{2\eta}\|x_{j+1}-x_j\|^2.
\]
This inequality, the definition of $\tx_j$ in \eqref{def:tx}, and the definition of $\delta_j$ in \eqref{def:tj} imply that
\begin{align*}
    \delta_{j+1} &= f^\eta(\tx_{j+1}) - f_{j+1}^\eta(x_{j+1})
    \le f^\eta(\tx_j) - f_j^\eta(x_j) -\frac{1}{2\eta}\|x_{j+1}-x_j\|^2\\
    &= \delta_j -\frac{1}{2\eta}\|x_{j+1}-x_j\|^2.
\end{align*}

d) Using the second inequality in (a), (b) with $j=J$, and the first inequality in (c), we have
  \begin{align*}
    f(\tx_J) &- f(x) + \frac{1}{2\eta}\|x-x_J\|^2 
     \stackrel{\text{(a)}}\le f(\tx_J) - f_J(x) + \frac{1}{2\eta}\|x-x_J\|^2 \\
    & \stackrel{\text{(b)}}\le f(\tx_J) - f_J^\eta(x_J) + \frac{1}{2\eta}\|x-y\|^2
     \stackrel{\text{(c)}}\le \delta - \frac{1}{2\eta}\|\tx_J-y\|^2 + \frac{1}{2\eta}\|x-y\|^2.
\end{align*}
This statement then follows from rearranging the terms and the definition of $f^\eta$ in \eqref{eq:subproblem}.

e) This statement directly follows from the first-order optimality condition of \eqref{eq:subproblem}.
\end{proof}

Clearly, when Algorithm \ref{alg:PBS} terminates, the output $\tilde x_J$ is a $\delta$-solution to \eqref{eq:subproblem}. To see this, note that, using the first inequality in Lemma \ref{lem:bundle}(c), \eqref{def:xj}, and the fact that $f_J^\eta(\cdot)\le f^\eta(\cdot)$, we have
    \[
        f^\eta(\tilde x_J) \le \delta + f_J^\eta(x_J) \overset{\eqref{def:xj}}{\le} \delta + f_J^\eta(x^*) \le \delta + f^\eta(x^*).
    \]
It is also easy to see that $\delta_j$ is computable upper bound on the gap $f^\eta(\tx_j) - f^\eta(x^*)$. Hence, Algorithm \ref{alg:PBS} terminates when $\delta_j \le \delta$. 

\subsection{Complexity for Hölder Smooth Optimization}\label{subsec:bundle}

This subsection is devoted to the complexity analysis of Algorithm \ref{alg:PBS} for solving \eqref{eq:subproblem} where $f$ is Hölder smooth, i.e., satisfying \eqref{ineq:semi-smooth}.
The following lemma provides basic recursive formulas and is the starting point of the analysis of Algorithm~\ref{alg:PBS}.

\begin{lemma}\label{lem:tj}
Assume $f$ is convex and $L_\alpha$-Hölder smooth.
Then, for every $j\ge 1$, the following statements hold:
\begin{itemize}
    \item[a)] $\delta_j \le \frac {L_\alpha}{\alpha+1} \|x_j-x_{j-1}\|^{\alpha+1}$; 
    \item[b)] $\delta_{j+1} + \frac{1}{2\eta} \left(\frac{\alpha+1}{L_\alpha} \delta_{j+1}\right)^{\frac{2}{\alpha+1}} \le \delta_j$.
\end{itemize}
\end{lemma}

\begin{proof}
a) It follows from the definition of $\delta_j$ in \eqref{def:tj} and the definition of $\tx_j$ in \eqref{def:tx} that
\begin{align*}
    \delta_j & \stackrel{\eqref{def:tj}}=f^\eta(\tx_j) - f_j^\eta(x_j) \stackrel{\eqref{def:tx}} \le f^\eta(x_j) - f_j^\eta(x_j) = f(x_j)-f_j(x_j) \\
    &\le f(x_j) - f(x_{j-1}) - \inner{f'(x_{j-1})}{x_j-x_{j-1}} \\
    &\le \frac{L_\alpha}{\alpha+1} \|x_j-x_{j-1}\|^{\alpha+1},
\end{align*}
where the second inequality is due to the definition of $f_j$ in the step 2 of Algorithm \ref{alg:PBS}, and the third inequality is due to \eqref{ineq:semi} with $(u,v)=(x_j,x_{j-1})$.

b) This statement directly follows from a) and the second inequality in Lemma \ref{lem:bundle}(c).
\end{proof}

We know from Lemma \ref{lem:bundle}(c) that $\{\delta_j\}_{j\ge1}$ is non-increasing.
The next proposition gives a bound on $j$ so that $\delta_j\le \delta$, i.e., the termination criterion in step 4 of Algorithm \ref{alg:PBS} is satisfied.

\begin{proposition}\label{prop:tj}
Define
\begin{equation}\label{def:beta}
    \beta := \frac{1}{2\eta} \left(\frac{\alpha+1}{L_\alpha} \right)^{\frac{2}{\alpha+1}} \delta^{\frac{1-\alpha}{\alpha+1}}, \quad 
    j_0=1+\left\lceil\frac{1+\beta}{\beta} \log\left(\frac{\delta_1}{\delta}\right)\right\rceil.
\end{equation}
Then, the following statements hold:
\begin{itemize}
    \item[a)] if $\delta_j>\delta$, then $(1+\beta)\delta_j\le 
    \delta_{j-1}$;
    \item[b)] $\delta_j\le \delta$ for every $j\ge j_0$.
\end{itemize}
As a consequence, the iteration count $J$ in Algorithm \ref{alg:PBS} satisfies $J\le j_0$.
\end{proposition}

\begin{proof}
a) Using the definition of $\beta$ in \eqref{def:beta}, the assumption that $\delta_j>\delta$, and Lemma \ref{lem:tj}(b), we obtain
\[
(1+\beta) \delta_j = \delta_j + \frac{1}{2\eta} \left(\frac{\alpha+1}{L_\alpha} \right)^{\frac{2}{\alpha+1}} \delta^{\frac{1-\alpha}{\alpha+1}} \delta_j
\le \delta_j + \frac{1}{2\eta} \left(\frac{\alpha+1}{L_\alpha} \delta_j\right)^{\frac{2}{\alpha+1}}
\le \delta_{j-1}.
\]

b) Since $\{\delta_j\}_{j\ge 1}$ is non-increasing, it suffices to prove that $\delta_{j_0}\le\delta$.
We prove this statement by contradiction. Suppose that $\delta_{j_0}>\delta$, then we have $\delta_j>\delta$ for $j\le j_0$.
Hence, statement (a) holds for $j\le j_0$. Using this conclusion repeatedly
and the fact that $\tau \le \exp(\tau -1)$ with $\tau=1/(1+\beta)$, we have
\[
\delta_{j_0}\le \frac{1}{(1+\beta)^{j_0-1}} \delta_1 \le \exp\left(-\frac{\beta}{1+\beta}(j_0-1)\right) \delta_1 \le \delta,
\]
where the last inequality is due to the definition of $j_0$ in \eqref{def:beta}.
This contradicts with the assumption that $\delta_{j_0}>\delta$, and hence we prove this statement.
\end{proof}

The following result shows that $\delta_1$ is bounded from above, and hence the bound in Proposition~\ref{prop:tj} is meaningful.

\begin{lemma}\label{lem:t1}
For a given $y\in \R^d$, we have
\[
\delta_1 \le \frac {L_\alpha \eta^{\alpha+1}}{\alpha+1} \|f'(y)\|^{\alpha+1}.
\]
\end{lemma}

\begin{proof}
Following the optimality condition of \eqref{def:xj} with $j=1$,
we have $x_0-x_1 = \eta f'(x_0) = \eta f'(y)$.
This identity and Lemma \ref{lem:tj}(a) with $j=1$ then imply that the lemma holds.
\end{proof}

We now conclude the iteration-complexity bound for Algorithm \ref{alg:PBS}. 

\begin{theorem}\label{thm:bundle}
Algorithm \ref{alg:PBS} takes $\tilde {\cal O}\left(\eta L_\alpha^{\frac{2}{\alpha+1}} \left(\frac{1}{\delta}\right)^{\frac{1-\alpha}{\alpha+1}} + 1\right)$ iterations to terminate.
\end{theorem}

\begin{proof}
    This theorem follows directly from Proposition \ref{prop:tj} and Lemma \ref{lem:t1}.
\end{proof}

\subsection{Complexity for Hybrid Optimization}\label{subsec:composite-OPT}


This subsection is devoted to the complexity analysis of Algorithm \ref{alg:PBS} for solving \eqref{eq:subproblem} where $f$ is a hybrid function satisfying \eqref{ineq:composite}.
The following lemma is an analogue of Lemma \ref{lem:tj} and provides key recursive formulas for $\delta_j$, which is defined in \eqref{def:tj}.

\begin{lemma}\label{lem:tj1}
Assume $f$ is convex and satisfies \eqref{ineq:composite}. For $\delta>0$, define
\begin{equation}\label{def:M}
    M= \sum_{i=1}^n \frac{L_{\alpha_i}^{\frac{2}{\alpha_i+1}} }{ [(\alpha_i+1)\delta]^{\frac{1-\alpha_i}{\alpha_i+1}} }.
\end{equation}
Then, for every $j\ge 1$, the following statements hold:
\begin{itemize}
    \item[a)] $\delta_j \le \frac M2 \|x_j-x_{j-1}\|^2 + \sum_{i=1}^n(1-\alpha_i) \frac{\delta}2$; 
    \item[b)] $\left(1+\frac{1}{\eta M}\right)\left(\delta_{j+1} - \sum_{i=1}^n(1-\alpha_i)\frac{\delta}{2}\right)\le \delta_j- \sum_{i=1}^n(1-\alpha_i) \frac{\delta}{2}$.
\end{itemize}
\end{lemma}

\begin{proof}
a) Following a similar argument as in the proof of Lemma \ref{lem:tj}(a) with \eqref{ineq:semi} replaced by \eqref{ineq:hybrid}, we have
\begin{equation}\label{ineq:tj1}
    \delta_j \le \sum_{i=1}^n \frac{L_{\alpha_i}}{\alpha_i+1} \|u-v\|^{\alpha_i+1}.
\end{equation}
Using the Young's inequality $ab\le a^p/p + b^q/q$ with 
	\[
	a = \frac{L_\alpha}{(\alpha+1)\delta^{\frac{1-\alpha}{2}}}\|x_j-x_{j-1}\|^{\alpha+1}, \quad b= \delta^{\frac{1-\alpha}{2}}, \quad p= \frac{2}{\alpha+1}, \quad q= \frac{2}{1-\alpha}, 
	\]
	we obtain
	\[
	\frac{L_\alpha}{\alpha+1}\|x_j-x_{j-1}\|^{\alpha+1} \le
	\frac{L_\alpha^{\frac{2}{\alpha+1}} }{2 [(\alpha+1)\delta]^{\frac{1-\alpha}{\alpha+1}} }\|x_j-x_{j-1}\|^2
	+ \frac{(1-\alpha)\delta}{2}.
	\]
Combining the above inequality and \eqref{ineq:tj1}, and using the definition of $M$ in \eqref{def:M}, we prove the statement.

b) It immediately follows from (a) and the second inequality in Lemma~\ref{lem:bundle}(c) that
    \[
    \delta_{j+1} + \frac{1}{\eta M}\left(\delta_{j+1}-\sum_{i=1}^n(1-\alpha_i) \frac{\delta}{2}\right) \le \delta_{j+1} + \frac{1}{2\eta}\|x_{j+1}-x_j\|^2 \le \delta_j,
    \]
    and hence the statement follows.
\end{proof}

The following lemma gives an upper bound on $\delta_1$ similar to Lemma \ref{lem:t1}.
	
\begin{lemma}\label{lem:t1-hybrid}
For a given $y\in \R^d$, we have
\[
\delta_1 \le \sum_{i=1}^n \frac{L_{\alpha_i}}{\alpha_i+1} \|f'(y)\|^{\alpha_i+1}.
\]
\end{lemma}

\begin{proof}
This lemma follows from a similar argument as in the proof of Lemma~\ref{lem:t1}.
\end{proof}

The following proposition is the key result in establishing the iteration-complexity of Algorithm~\ref{alg:PBS}.

\begin{proposition}\label{prop:null} 
		We have $\delta_j\le \delta$, for every $j$ such that
		\begin{equation}\label{ineq:j-hybrid}
		    j \ge (1+ \eta M ) \log\left( \frac{2  \delta_1}{\delta}\right).
		\end{equation}
	\end{proposition}
    
	\begin{proof}
        Let 
        \begin{equation}\label{def:tau}
            \tau = \frac{\eta M}{1+\eta M},
        \end{equation}
        then Lemma \ref{lem:tj1}(b) becomes
        \[
        \delta_{j+1} - \sum_{i=1}^n(1-\alpha_i)\frac{\delta}{2} \le \tau \left(t_{j} - \sum_{i=1}^n(1-\alpha_i)\frac{\delta}{2}\right).
        \]
	    Using the above inequality repeatedly and the fact that $\tau \le \exp(\tau-1)$, we have for every $j \ge 1$,
		\[
		\delta_j - \frac{(1-\alpha)\delta}2 \le \tau^{j-1}\left(\delta_1 - \frac{(1-\alpha)\delta}2 \right) \le \tau^{j-1}\delta_1 \le \exp\{(\tau-1) (j-1)\} \delta_1. 
		\]
		Hence, it is easy to see that $\delta_j\le \delta$ if $j\ge \frac{1}{1-\tau} \log\left( \frac{2  \delta_1}{\delta}\right)$.
		Using the definition of $\tau$ in \eqref{def:tau}, we have if $j$ is as in \eqref{ineq:j-hybrid}, then $\delta_j\le \delta$.
    \end{proof}

We are ready to present the complexity bound for Algorithm \ref{alg:PBS}. 

\begin{theorem}\label{thm:bundle-hybrid}
Algorithm \ref{alg:PBS} takes $\tilde {\cal O}\left(\eta M + 1\right)$ iterations to terminate, where $M$ is as in \eqref{def:M}.
\end{theorem}

\begin{proof}
    This theorem follows directly from Proposition \ref{prop:null} and Lemma \ref{lem:t1-hybrid}.
\end{proof}

\subsection{Implementation of Algorithm~\ref{alg:PBS}}

This subsection presents the simulation results of Algorithm~\ref{alg:PBS} on solving the regularized subproblem \eqref{eq:subproblem} for two objective functions $f$: quadratic programming (QP) and $\ell_p$ regression.
In both cases, the subgradient $f'$ is computed by automatic differentiation via Zygote.jl~\cite{innes2018don}, and the subproblem \eqref{def:xj} is reformulated as a QP and solved using Clarabel.jl~\cite{goulart2024clarabel}. 
Numerical simulations are conducted on an i9-13900k desktop with 64 GB of RAM

\textit{Quadratic Programming}
We first consider the unconstrained QP problem 
\[f(x)=\frac{1}{2}x^\top Q x + \inner{c}{x}\]
where $Q\in \mathbb{S}^{d}_+$ and $c\in\R^d$. We generate $Q = AA^\top/\|AA^\top\|_\infty$
where $A\in\R^{d\times d}$ has normally distributed entries and $\|AA^\top\|_\infty=\max_{ij}|(AA^\top)_{ij}|$ is the entrywise infinity norm. The linear term $c$ and point $y$ are also entrywise normally distributed. The dimension $d$ is set to be $1000$. 

Note that \eqref{eq:subproblem} for QP has the closed-form solution
\[
   \underset{x\in\R^n}\argmin\left\{\frac{1}{2} x^\top Q x + \inner{c}{x} + \frac{1}{2\eta}\|x-y\|^2\right\}=(Q + \eta^{-1} I)^{-1}(\eta^{-1} y-c),
\]
hence we can compare the progress of Algorithm~\ref{alg:PBS} against the true minimum. We run Algorithm~\ref{alg:PBS} until the condition $\delta_j<10^{-6}$ is satisfied. Fig.~\ref{fig:qp_experiment} shows the function value decrease of the minimum value iterate $f_y^\eta(\tx_j)$ versus the optimal value $f_y^\eta(x^*)$ with varying $\eta$.
Noting that Algorithm~\ref{alg:PBS} requires more iterations as $\eta$ increases, this observation is consistent with the complexity bound (proportional to $\eta$) stated in Theorem~\ref{thm:bundle}.

\begin{figure}
    \centering
    \includegraphics[width=0.5\linewidth]{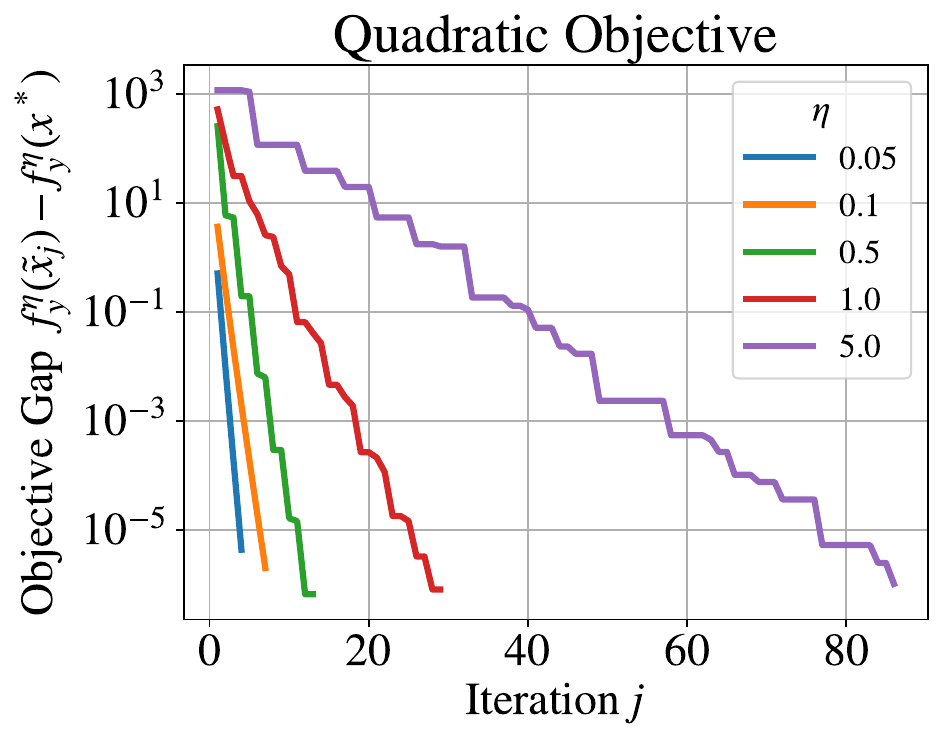}
    \caption{Proximal subproblem progress of Algorithm~\ref{alg:PBS} in quadratic programming. }
    \label{fig:qp_experiment}
\end{figure}

\textit{$\ell_p$ Regression}
We next consider $\ell_p$ regression, where the objective $f$ is of the form
\[
    f(x) = \|Ax-b\|_p^p,
\]
where $A\in \R^{n\times d}$ and $b\in\R^n$ have normally distributed entries and $A$ is again divided by its entrywise infinity norm. We set $d=100$ and $n=500$ for testing. The point $y\in \R^d$ is entrywise normally distributed, and is identical for all $p$ values tested. Algorithm~\ref{alg:PBS} is terminated when $\delta_j<10^{-6}$. Fixing $\eta=1.0$, Fig.~\ref{fig:lp_regression} shows the trajectory of the gap $\delta_j$ for varying $p$ values.

\begin{figure}[H]
    \centering
    \includegraphics[width=0.5\linewidth]{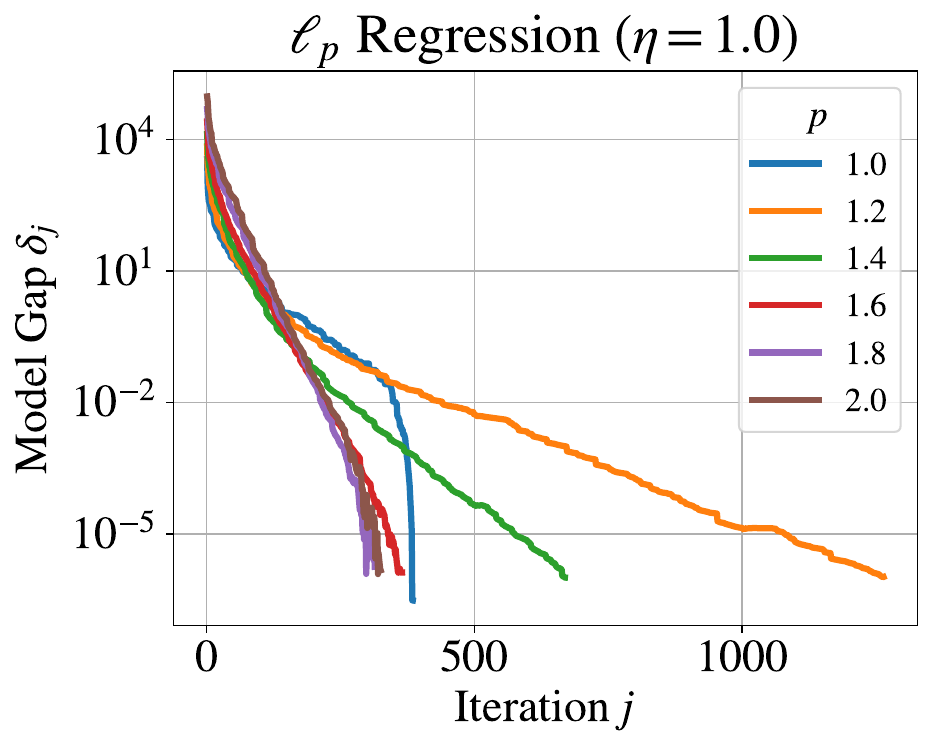}
    \caption{Proximal subproblem progress of Algorithm~\ref{alg:PBS} in $\ell_p$ regression. }
    \label{fig:lp_regression}
\end{figure}

\section{Adaptive Proximal Bundle Method}\label{sec:APBM}

As discussed in Section \ref{sec:opt}, the cutting-plane method (i.e., Algorithm \ref{alg:PBS}) is widely used in the proximal bundle method as a subroutine to repeatedly solve the proximal subproblem \eqref{eq:prox-sub}.
Since the proximal bundle method uses a more accurate cutting-plane model $f_j$ rather than a linearization as an approximation of the objective function $f$, it generalizes the subgradient method and is able to work with weaker regularization, namely larger stepsize $\eta$. This explains why both methods have optimal complexity bounds \cite{liang2024unified,liang2020proximal}, but the proximal bundle method is always more efficient in practice.

For both the subgradient method and the proximal bundle method to have the optimal performance, one needs to carefully select the stepsize $\eta$, namely, being small enough for the subgradient method and within a certain (but relatively large) range for the proximal bundle method. In both cases, we need to know the problem-dependent parameters such as $\alpha$ and $L_\alpha$, which are unknown or hard to estimate in practice.
In this section, we develop the APBM based on an adaptive stepsize strategy for \eqref{eq:prox-sub} and using Algorithm \ref{alg:PBS} to solve each subproblem \eqref{eq:prox-sub}.
We also discuss variants of adaptive subgradient methods and compare them with APBM.
For simplicity, we only present the analysis of Hölder smooth functions satisfying \eqref{ineq:semi-smooth}, while the hybrid functions satisfying \eqref{ineq:composite} can be similarly analyzed using results from Subsection~\ref{subsec:composite-OPT}.

From practical observations \cite{liang2024single}, the proximal bundle method works well when the number of inner iterations (i.e., those of Algorithm \ref{alg:PBS} to solve \eqref{eq:prox-sub}) stays as a constant much larger than 1 (i.e., that of the subgradient method), say $10$. Recall from Theorem \ref{thm:bundle} that inner complexity is $\tilde {\cal O}\left(\eta L_\alpha^{\frac{2}{\alpha+1}} \left(\frac{1}{\delta}\right)^{\frac{1-\alpha}{\alpha+1}} + 1\right)$. Since we do not know $\alpha$ and $L_\alpha$, we cannot choose a constant stepsize $\eta$ so that the number of inner iterations is close to a desired number such as $10$. Hence, an adaptive stepsize rule is indeed needed.

By carefully examining Proposition \ref{prop:tj} and Theorem \ref{thm:bundle}, we find that the inner complexity is $\tilde {\cal O}(\beta^{-1}+1)$ where $\beta$ is as in \eqref{def:beta}. Suppose we want to prescribe the number of inner iterations to be close to $\beta_0^{-1}$ for some $\beta_0\in(0,1]$, if $\beta_0 \le \beta$, then by Proposition \ref{prop:tj}(a), we have
\begin{equation}\label{ineq:test}
(1+\beta_0) \delta_{j} \le \delta_{j-1}.
\end{equation}
Hence, it suffices to begin with a relatively large $\eta$, check \eqref{ineq:test} to determine whether the $\eta$ is small enough (i.e., $\beta$ is large enough), and adjust $\eta$ (if necessary) by progressively halving it.

Algorithm \ref{alg:PB} below is a formal statement of APBM based on the above intuition.

\begin{algorithm}[H]
	\caption{Adaptive Proximal Bundle Method} 
	\label{alg:PB}
	\begin{algorithmic}
	\REQUIRE Let $y_0\in \R^d$, $\eta_0>0$, $\beta_0\in(0,1]$, and $\varepsilon>0$  be given.
    \FOR{$k=1,2,\cdots$}
        \STATE 
		Call Algorithm \ref{alg:PBS} with $(y,\eta,\delta)=(y_{k-1},\eta_{k-1},\varepsilon/2)$ and output $(y_k,\tilde y_k)=(x_J,\tilde x_J)$.
        \IF {\eqref{ineq:test}
        is always true in the execution of Algorithm \ref{alg:PBS},}
            \STATE set $\eta_k=\eta_{k-1}$;
        \ELSE
            \STATE set $\eta_k=\eta_{k-1}/2$.
        \ENDIF
        \ENDFOR
	\end{algorithmic}
\end{algorithm}

The following lemma provides basic results of Algorithm \ref{alg:PBS} and is the starting point of the analysis of APBM.

\begin{lemma}\label{lem:APBM}
Assume $f$ is convex and $L_\alpha$-Hölder smooth. The following statements hold for APBM:
\begin{itemize}
    \item[a)] for every $k\ge 1$ and $u\in \R^d$, we have
    \begin{equation}\label{ineq:iteration}
        2\eta_{k-1}[f(\tilde y_k) - f(u)] \le \|y_{k-1}-u\|^2 - \|y_k - u\|^2 + \eta_{k-1} \varepsilon;
    \end{equation}
    \item[b)] for any $k\ge 1$, if 
    \begin{equation}\label{ineq:eta}
        \eta_{k-1} \le \frac{1}{2\beta_0} \left(\frac{\alpha+1}{L_\alpha} \right)^{\frac{2}{\alpha+1}} \left(\frac{\varepsilon}{2}\right)^{\frac{1-\alpha}{\alpha+1}},
    \end{equation}
then $\eta_k=\eta_{k-1}$;
    \item[c)] $\{\eta_k\}$ is a non-increasing sequence;
    \item[d)] for every $k\ge 0$, 
        \begin{equation}\label{ineq:etak}
        \eta_k \ge \underline{\eta}:= \min \left \lbrace \frac{1}{4\beta_0} \left(\frac{\alpha+1}{L_\alpha} \right)^{\frac{2}{\alpha+1}} \left(\frac{\varepsilon}{2}\right)^{\frac{1-\alpha}{\alpha+1}}, \eta_0 \right\rbrace.
        \end{equation}
\end{itemize}
\end{lemma}
\begin{proof}
    a) It follows from Lemma \ref{lem:bundle}(d) that for every $u \in \R^d$
    \[
    2\eta[f(\tilde x_J) - f(x)] \le 2\eta\delta + \|u-y\|^2 - \|x_J - u\|^2.
    \]
    Noting from step 2 of Algorithm \ref{alg:PB} that $(\delta,\eta,y,x_J,\tilde x_J)=(\varepsilon/2,\eta_{k-1},y_{k-1},y_k,\tilde y_k)$, which together with the above inequality, implies that \eqref{ineq:iteration} holds.
    
    b) It follows from Proposition \ref{prop:tj}(a) and \eqref{ineq:eta} that \eqref{ineq:test} always holds in the execution of Algorithm \ref{alg:PBS}. In view of step 3 of Algorithm \ref{alg:PB}, there holds $\eta_k=\eta_{k-1}$.

    c) This statement clearly follows from step 3 of Algorithm \ref{alg:PB}.

    d) This statement immediately follows from (b) and step 3 of Algorithm~\ref{alg:PB}.
\end{proof}

The following theorem gives the total iteration-complexity of APBM. 

\begin{theorem}
    Assume $f$ is convex and $L_\alpha$-Hölder smooth. If $\eta_0\le \|y_0-x_*\|^2/\varepsilon$, then the iteration-complexity to obtain an $\varepsilon$-solution to \eqref{eq:opt} (i.e., a point $\hat x$ such that $f(\hat x)-\min_{x\in \R^d} f(x) \le \varepsilon$) is given by
    \begin{equation}\label{bound:total}
        \tilde {\cal O} \left( \frac{L_\alpha^{\frac{2}{\alpha+1}} \|y_0-x_*\|^2}{\varepsilon^{\frac{2}{\alpha+1}}}
        + \eta_0 L_\alpha^{\frac{2}{\alpha+1}} \left(\frac{1}{\varepsilon}\right)^{\frac{1-\alpha}{\alpha+1}}
        \log\left(\frac{\eta_0}{\underline{\eta}}\right) + 1 \right)
    \end{equation}
    where $\underline{\eta}$ is as in \eqref{ineq:etak}.
\end{theorem}

\begin{proof}
Noting from Lemma \ref{lem:APBM} (d) that $\eta_k$ is bounded from below, we know there exists some $\tilde \eta \in [\underline{\eta}, \eta_0]$ such that for some $k_0\ge 1$, $\eta_k\equiv \tilde \eta$ for $k\ge k_0$. Thus, it follows from the assumption that $\eta_0\le \|y_0-x_*\|^2/\varepsilon$ that
    \begin{equation}\label{cond:eta}
        \underline{\eta} \le \tilde \eta \le \frac{\|y_0-x_*\|^2}{\varepsilon}.
    \end{equation}
We consider the worst-case scenario where APBM keeps halving the stepsize until it is stable at $\tilde \eta$ and the convergence relies on the conservative stepsize $\tilde \eta$.
Summing \eqref{ineq:iteration} from $k=1$ to $n$, we have
    \begin{align*}
        2 \sum_{k=1}^n \eta_{k-1} \left(\min_{1\le k \le n} f(\tilde y_k) - f(u)\right) &\le 2 \sum_{k=1}^n \eta_{k-1}[f(\tilde y_k) - f(u)] \\
        &\le \|y_0-u\|^2 - \|y_n-u\|^2 + \varepsilon \sum_{k=1}^n\eta_{k-1}.
    \end{align*}
The above inequality with $u=x_*$, the fact that $\eta_k \le \eta_0$, and the assumption that $\eta_k\equiv \tilde \eta$ for $k\ge k_0$ imply that
\begin{equation}\label{ineq:dist}
    \|y_n-x_*\|^2 \le \|y_0-x_*\|^2 + n \eta_0 \varepsilon 
\end{equation}
and
\[
\min_{1\le k \le n} f(\tilde y_k) - f_* \le \frac{\|y_0-x_*\|^2}{2 \sum_{k=1}^n \eta_{k-1}} + \frac{\varepsilon}{2} \le \frac{\|y_0-x_*\|^2}{2(n-k_0) \tilde \eta} + \frac{\varepsilon}{2}.
\]
In order to have $\min_{1\le k \le n} f(\tilde y_k) - f_* \le \varepsilon$, we need
\begin{equation}\label{ineq:n-k0}
    n-k_0= {\cal O} \left(\frac{\|y_0-x_*\|^2}{\tilde \eta\varepsilon}+1\right).
\end{equation}
Moreover, it follows from the way $\eta_k$ is updated in step 3 and Lemma \ref{lem:APBM}(d) that
\begin{equation}\label{ineq:k0}
    k_0 = {\cal O}\left( \log\left(\frac{\eta_0}{\tilde\eta}\right) + 1\right) = {\cal O}\left( \log\left(\frac{\eta_0}{\underline{\eta}}\right) + 1\right).
\end{equation}
Indeed, \eqref{ineq:dist} holds with $n$ replaced by any $k \le n$ and
\[
\|y_k-x_*\|^2 \le \|y_0-x_*\|^2 + n \eta_0 \varepsilon.
\]
It thus follows from \eqref{ineq:n-k0} and \eqref{ineq:k0} that $\{y_k\}$ is bounded.
As a result, using Lemma \ref{lem:t1}, we can derive a uniform bound on $\delta_1$ for every call to Algorithm~\ref{alg:PBS}.
Now, using Theorem \ref{thm:bundle}, we have the iteration-complexity of every call to Algorithm \ref{alg:PBS} is uniformly bounded by
\begin{equation}\label{bnd1}
    \tilde {\cal O}\left(\tilde \eta L_\alpha^{\frac{2}{\alpha+1}} \left(\frac{1}{\varepsilon}\right)^{\frac{1-\alpha}{\alpha+1}} +1 \right)
\end{equation}
for every cycle $k\ge k_0$ and by
\begin{equation}\label{bnd2}
    \tilde {\cal O}\left(\eta_0 L_\alpha^{\frac{2}{\alpha+1}} \left(\frac{1}{\varepsilon}\right)^{\frac{1-\alpha}{\alpha+1}} +1  \right)
\end{equation}
for every cycle $k\le k_0-1$.
Hence, multiplying \eqref{ineq:n-k0} and \eqref{bnd1} and using \eqref{cond:eta} and the definition of $\underline{\eta}$ in \eqref{ineq:etak}, we obtain the iteration-complexity  
\[
\tilde {\cal O}\left(\frac{L_\alpha^{\frac{2}{\alpha+1}} \|y_0-x_*\|^2}{\varepsilon^{\frac{2}{\alpha+1}}} +1 \right)
\]
for cycles $k\ge k_0$,
and multiplying \eqref{ineq:k0} and \eqref{bnd2}, we obtain the iteration-complexity  
\[
\tilde {\cal O}\left(\eta_0 L_\alpha^{\frac{2}{\alpha+1}} \left(\frac{1}{\varepsilon}\right)^{\frac{1-\alpha}{\alpha+1}}
        \log\left(\frac{\eta_0}{\underline{\eta}}\right) +1 \right)
\]
for cycles $k\le k_0-1$.
Finally, the total iteration-complexity \eqref{bound:total} clearly follows from the above two bounds.
\end{proof}

We note that the final $\varepsilon$-solution produced by Algorithm~\ref{alg:PB} is the point $\tilde y_k$ that achieves $\min_{1 \le k \le n} f(\tilde{y}_k)$. This differs from the last-iterate convergence observed in the smooth case, since the objective function $f$ here is Hölder smooth and may include the nonsmooth case (i.e., $\alpha = 0$).

\textbf{Discussion on other universal methods}
Several universal methods based on the backtracking line-search procedure have been studied in the literature. 
Paper \cite{nesterov2015universal} considers the same Hölder smooth problem (with an additional hybrid function $h$) as in this paper. 
To finds an $\varepsilon$-solution of \eqref{eq:opt}, the universal primal gradient method proposed in \cite{nesterov2015universal} starts from an initial pair $(\hat x_0,\eta_0)$ and in the $(j+1)$-th iteration
searches for a pair $(x_\eta,\eta)$ satisfying a condition 
\begin{equation}\label{ineq:condition}
    f(x_\eta)-\ell_f(x_\eta;\hat x_{j}) - \frac{1}{2\eta}\|x_\eta-\hat x_{j}\|^2 \le \frac{\varepsilon}2,
\end{equation}
where $\ell_f(u;v) = f(v) + \inner{f'(v}{u-v}$ and 
\begin{equation}\label{def:xlam}
    x_\eta =\underset{u\in  \R^n}\argmin
		\left\lbrace  \ell_f(u;\hat x_{j})+h(u) +\frac{1}{2\eta} \|u- \hat x_{j} \|^2 \right\rbrace.
\end{equation}
If the condition \eqref{ineq:condition} is not satisfied, then the method rejects the pair, sets $\eta \leftarrow \eta/2$, and updates $x_\eta$ as in \eqref{def:xlam} with the new $\eta$, otherwise, it accepts the pair and sets $(\hat x_{j+1},\eta_{j+1})=(x_\eta,\eta)$.
Two other universal methods are developed in \cite{nesterov2015universal}, namely, the universal dual gradient method and the universal fast gradient method. Following \cite{nesterov2015universal}, paper \cite{grimmer2024optimal} extends the universal fast gradient method to the case of hybrid functions \eqref{ineq:composite}.
Motivated by the bundle-level method of \cite{lemarechal1995new}, paper \cite{lan2015bundle} proposes two accelerated variants, i.e., the accelerated bundle-level method and the accelerated prox-level method.
The parallel bundle method of \cite{diaz2023optimal} is also shown to be universal at the price of running multiple threads.

Paper~\cite{liang2024unified} proposes an adaptive composite subgradient (A-CS) method for solving \eqref{eq:opt} where $f$ satisfies
     \begin{equation}\label{ineq:est}
	\|f'(x)-f'(y)\| \le 2M_f + L_f\|x-y\|, \quad \forall x,y \in \R^d.
	\end{equation}
It is shown in Proposition 2.1 of \cite{liang2024unified} that any function $f$ that satisfies
\[
\|f'(x)-f'(y)\| \le 2M_\alpha + L_\alpha\|x-y\|^\alpha, \quad \forall x,y \in \R^d,
\]
for some $\alpha \in [0,1]$ also satisfies \eqref{ineq:est} with
\[
M_f(\theta):=M_\alpha + \frac{L_\alpha \theta}{2}, \quad L_f (\theta):= L_\alpha \alpha\left(\frac{1-\alpha}{\theta}\right)^{\frac{1-\alpha}{\alpha}}
\]
for any $\theta>0$.
Hence, the Hölder smooth functions \eqref{ineq:semi-smooth} considered in this paper are included in the class of functions satisfying \eqref{ineq:est}.
More interestingly, a careful look at A-CS of \cite{liang2024unified} and the universal primal gradient method of \cite{nesterov2015universal} reveals that the two methods are identical. 


The universal primal gradient method is essentially an adaptive subgradient method and the convergence of subgradient methods relies on small enough stepsizes, so it is natural to enforce \eqref{ineq:condition} to make the method adaptive.
However, the bundle method converges with any constant stepsize $\eta$ since it guarantees the condition $\delta_j\le \varepsilon/2$, which is in the same spirit of \eqref{ineq:condition}, by the cutting-plane approach (i.e., Algorithm \ref{alg:PBS}) but not by small $\eta$. Therefore, it is not necessary to use a small $\eta$ in every iteration of each call to Algorithm~\ref{alg:PBS}. Instead of frequently reducing $\eta$, by the introduction of $\beta_0$, APBM develops a way to regulate the complexity of Algorithm \ref{alg:PBS} and adjust $\eta$ only when \eqref{ineq:test} is not always true in the previous call to Algorithm~\ref{alg:PBS}.
Another difference between APBM and the universal primal gradient method is that the latter rejects all the pairs $(x_\eta,\eta)$ until \eqref{ineq:condition} is satisfied, but APBM always accepts the output of Algorithm~\ref{alg:PBS} even if \eqref{ineq:test} is not true for every iteration in Algorithm~\ref{alg:PBS}.
Therefore, APBM potentially employs a larger stepsize $\eta$ than the universal primal gradient method and is thus a more relaxed adaptive method.

\vgap

{\bf Discussion on optimal universal methods}
The lower complexity bound for solving \eqref{eq:opt} is shown in \cite{nemirovsky1983problem} to be
\[
{\cal O}\left(\left(\frac{L_\alpha \|y_0-x_*\|^{1+\alpha}}{\varepsilon}\right)^{\frac{2}{1+3\alpha}}\right). 
\]
The well-known Nesterov's accelerated gradient method has been shown in \cite{nemirovskii1985optimal} to match the above complexity bound and hence is an optimal method.
The accelerated bundle-level method of \cite{lan2015bundle}, the universal fast gradient method of \cite{nesterov2015universal}, and a follow-up work \cite{grimmer2024optimal} all establish optimal complexity bounds.

On the other hand, the dominant term of bound \eqref{bound:total} is its first term and it is only optimal when $\alpha=0$, i.e., $f$ is $L_0$-Lipschitz continuous. Motivated by \cite{nemirovskii1985optimal}, it is possible to develop optimal universal methods based on the accelerated gradient method.
This requires accelerated schemes in both PPF and Algorithm \ref{alg:PBS}.
Paper \cite{MonteiroSvaiterAcceleration} proposes an accelerated variant of PPF, which is extended by \cite{bubeck2019near,jiang2019optimal,gasnikov2019optimal} to obtain optimal $p$-th order methods with convergence rate ${\cal O}(k^{-(3p+1)/2})$ for $p\ge 2$.

We finally note that this paper does not aim to develop the optimal complexity of universal methods; rather, it presents an interesting application of our analysis of Algorithm \ref{alg:PBS} in the context of universal methods.

\section{Proximal Sampling Algorithm}\label{sec:sampling}


Assuming the RGO in the ASF can be realized, the ASF exhibits remarkable convergence properties. It was shown in \cite{lee2021structured} that Algorithm \ref{alg:ASF} converges linearly when $f$ is strongly convex. This convergence result is recently improved in \cite{chen2022improved} under various weaker assumptions on the target distribution $\pi^X \propto \exp(-f)$. Below we present several convergence results established in \cite{chen2022improved} that will be used in this paper, under the assumptions that $\pi^X$ is log-concave, or satisfies the log-Sobolev inequality or Poincar\'e inequality (PI). 
Recall that a probability distribution $\nu$ satisfies PI with constant $C_{\rm PI} > 0$ ($1/C_{\rm PI}$-PI) if for any smooth bounded function $\psi : \R^d\to\R$,
\[
    \mE_\nu[(\psi-\mE_\nu(\psi))^2]
    \le C_{\rm PI} \mE_\nu[\norm{\nabla \psi}^2].
\]

To this end, for two probability distributions $\rho \ll \nu$, we denote by 
    \[
        H_\nu(\rho) := \int \rho \log\frac{\rho}{\nu},\quad \chi_\nu^2 (\rho):= \int \frac{\rho^2}{\nu}-1
    \]
the {\em KL divergence} and the {\em Chi-squared divergence}, respectively. 
We denote by $W_2$ the Wasserstein-2 distance
    \[
        W_2^2(\nu,\rho) := \min_{\gamma\in \Pi(\nu,\rho)}\int \|x-y\|^2 {\mathrm d}\gamma(x,y),
    \]
where $\Pi(\nu,\rho)$ represents the set of all couplings between $\nu$ and $\rho$.

\begin{theorem}[{\cite[Theorems 2 \& 4]{chen2022improved}}]\label{Thm:LC}
We denote by $\rho_k^X$ the law of $x_k$ of Algorithm \ref{alg:ASF} starting from any initial distribution $\rho_0^X$. Then, the following statements hold:
\begin{itemize}
    \item[a)] if $\pi^X \propto \exp(-f)$ is log-concave (\textit{i.e.}, $f$ is convex), then $H_{\pi^X}(\rho_k^X)
        \le W_2^2(\rho_0^X, \pi^X)/(k\eta)$;
    \item[b)] if $\pi^X \propto \exp(-f)$ satisfies $\lam$-PI, then $\chi_{\pi^X}^2(\rho^X_k) \le \chi_{\pi^X}^2(\rho^X_0)/{(1 + \lam \eta)}^{2k}$.
\end{itemize}
\end{theorem}



As discussed earlier, to use ASF in sampling problems, we need to realize the RGO with efficient implementations. In the rest of this section, we develop efficient algorithms for RGO associated with the two scenarios of sampling we are interested in, and then combine them with the ASF to establish a proximal algorithm for sampling. The complexity of the proximal algorithm can be obtained by combining the above convergence results for ASF and the complexity results we develop for RGO. The rest of the section is organized as follows. In Subsection~\ref{sec:RGO} we develop an efficient algorithm for RGO associated with Hölder smooth potentials via rejection sampling. This is combined with ASF to obtain an efficient sampling algorithm from Hölder smooth potentials. In Subsection~\ref{sec:composite}, we further extend results to the second setting, i.e., hybrid potentials.

\subsection{Sampling from Hölder Smooth Potentials}\label{sec:RGO}

The bottleneck of using the ASF (Algorithm \ref{alg:ASF}) in sampling tasks with general distributions is the availability of RGO implementations.
In this subsection, we address this issue for convex Hölder smooth potentials by developing an efficient algorithm for the corresponding RGO. 

Our algorithm of RGO for $f$ is based on rejection sampling. We use a special proposal, namely a Gaussian distribution centered at the $\delta$-solution of \eqref{eq:subproblem}, which is obtained by invoking Algorithm \ref{alg:PBS}.
With this proposal and a sufficiently small $\eta>0$, the expected number of rejection sampling steps to obtain one effective sample turns out to be bounded from above by a dimension-free constant. 
To bound the complexity of the rejection sampling, we develop a novel technique to estimate a modified Gaussian integral (see Proposition~\ref{prop:key}).

To this end, let $J, \tilde x_J, x_J$ be the outputs of Algorithm \ref{alg:PBS} and define
\begin{subequations}\label{eq:newh}
\begin{align}
    h_1 &:= \frac{1}{2\eta}\|\cdot-x_J\|^2 + f_y^\eta(\tx_J) - \delta, \label{def:tf} \\
    h_2 &:=  \frac{1}{2\eta}\|\cdot-x^*\|^2 + \frac{L_\alpha}{\alpha+1}\|\cdot-x^*\|^{\alpha+1} + f_y^\eta(x^*). \label{def:h}
\end{align}
\end{subequations}
Note that $h_2$ is only used for analysis and thus the fact it depends on $x^*$ is not an issue.
Algorithm~\ref{alg:RGO-bundle} describes the implementation of RGO for $f$ based on Algorithm~\ref{alg:PBS} and rejection sampling. 

\begin{algorithm}[H]
	\caption{RGO Implementation based on Rejection Sampling}
	\label{alg:RGO-bundle}
	\begin{algorithmic}
	    \STATE 1. Let $y\in \R^d$, $\eta>0$, and $\delta>0$ be given, and run Algorithm \ref{alg:PBS} to compute $x_J$ and $\tx_J$.
		\STATE 2. Generate  $X\sim \exp(-h_1(x))$.
		\STATE 3. Generate $U\sim {\cal U}[0,1]$.
        \IF {$U \le \exp\left(-f_y^\eta(X)+h_1(X)\right)$,}
            \STATE accept/return $X$;
        \ELSE
            \STATE reject $X$ and go to step 2.
        \ENDIF
	\end{algorithmic}
\end{algorithm}

\begin{lemma}\label{lem:h1h2delta}
Assume $f$ is convex and $L_\alpha$-Hölder smooth. Let $f_y^\eta$ be as in \eqref{eq:subproblem} and $h_1$ and $h_2$ be as in \eqref{eq:newh}.
Then, for every $x\in \R^d$, we have
\begin{equation}\label{ineq:sandwich}
    h_1(x) \le f_y^\eta(x) \le h_2(x).
\end{equation}
\end{lemma}
\begin{proof}
The first inequality in \eqref{ineq:sandwich} immediately follows from Lemma \ref{lem:bundle}(d) and the definition of $h_1$ in \eqref{def:tf}.
By the definition of $f_y^\eta$ in \eqref{eq:subproblem} we get
\begin{align}
    f_y^\eta(x) - f_y^\eta(x^*)  
    = & f(x) - f(x^*)  + \frac{1}{2\eta}\|x-y\|^2 - \frac{1}{2\eta}\|x^*-y\|^2 \nn \\
    = & f(x) - f(x^*) + \frac{1}{2\eta}\|x-x^*\|^2 + \frac{1}{\eta}\inner{x-x^*}{x^*-y}. \label{eq:equal}
\end{align}
It follows from Lemma \ref{lem:bundle}(e) and \eqref{ineq:semi} with $(u,v)=(x,x^*)$ that
\[
f(x) - f(x^*) + \frac{1}{\eta}\inner{x^*-y}{x-x^*} \le \frac{L_\alpha}{\alpha+1}\|x-x^*\|^{\alpha+1},
\]
which together with \eqref{eq:equal} implies that
\[
f_y^\eta(x) - f_y^\eta(x^*) \le \frac{L_\alpha}{\alpha+1}\|x-x^*\|^{\alpha+1} + \frac{1}{2\eta}\|x-x^*\|^2.
\]
Using the above inequality and the definition of $h_2$ in \eqref{def:h}, we conclude that the second inequality in \eqref{ineq:sandwich} holds.
\end{proof}

From the expression of $h_1$ in \eqref{def:tf}, it is clear that the proposal distribution $\exp(-h_1(x))$ is a Gaussian centered at $x_J$. To achieve a tight bound on the expected runs of the rejection sampling, we use a function $h_2$ which is not quadratic; the standard choice of quadratic function does not give as tight results due to the lack of smoothness. To use this $h_2$ in the complexity analysis, we need to estimate the integral $\int \exp(-h_2)$, which turns out to be {\em a highly nontrivial task}. Below we establish a technical result on a modified Gaussian integral, which will be used later to bound the integral $\int \exp(-h_2)$ and hence the complexity of the RGO rejection sampling in Algorithm \ref{alg:RGO-bundle}.


\begin{proposition}\label{prop:key}
Let $\alpha\in[0,1]$, $\eta>0$, $a\ge 0$ and $d\ge 1$. If 
    \begin{equation}\label{eq:aeta}
        2a(\eta d)^{(\alpha+1)/2} \le 1,
    \end{equation}
then
\begin{equation}\label{ineq:int}
    \int_{\R^d} \exp\left(-\frac{1}{2\eta}\|x\|^2 - a\|x\|^{\alpha+1}\right) \rd x
\ge \frac{(2\pi\eta)^{d/2}}2.
\end{equation}
\end{proposition}

\begin{proof}
Denote $r=\|x\|$, then
\[
\rd x= r^{d-1} \rd r \rd S^{d-1},
\]
where $ \rd S^{d-1} $ is the surface area of the $ (d-1) $-dimensional unit sphere.
It follows that
\begin{align}
  \int_{\R^d} \exp\left(-\frac{1}{2\eta}\|x\|^2 - a\|x\|^{\alpha+1} \right) \rd x &= \int_0^\infty \int \exp\left(-\frac{1}{2\eta}r^2 - a r^{\alpha+1} \right) r^{d-1} \rd r \rd S^{d-1} \nn \\
&= \frac{2 \pi^{d/2}}{\Gamma\left( \frac d2 \right) } \int_0^\infty \exp\left( -\frac{1}{2\eta}r^2 - a r^{\alpha+1} \right) r^{d-1} \rd r.  \label{eq:int}
\end{align}
In the above equation, we have used the fact that the total surface area of a $ (d-1) $-dimensional unit sphere is $2 \pi^{d/2}/\Gamma\left( \frac d2 \right)$ where $\Gamma(\cdot)$ is the gamma function, i.e.,
\begin{equation}\label{eq:gamma}
        \Gamma(z)=\int_0^\infty t^{z-1} e^{-t} \rd t.
    \end{equation} 
Defining
\begin{equation}\label{def:F}
    F_{d,\eta}(a):= \int_0^\infty \exp\left( -\frac{1}{2\eta}r^2 - a r^{\alpha+1} \right) r^{d} \rd r,
\end{equation}
to establish \eqref{ineq:int},
it suffices to bound $F_{d-1,\eta}(a)$ from below.

It follows directly from the definition of $F_{d,\eta}$ in \eqref{def:F} that
\[
\frac{\rd F_{d-1,\eta}(a)}{\rd a}=\int_0^\infty \exp\left( -\frac{1}{2\eta}r^2 - a r^{\alpha+1} \right) (-r^{\alpha+1}) r^{d-1} \rd r 
=- F_{d+\alpha,\eta}(a).
\]
This implies $F_{d,\eta}$ is monotonically decreasing and thus $F_{d+\alpha,\eta}(a)\le F_{d+\alpha,\eta}(0)$. As a result,
    \[
        \frac{\rd F_{d-1,\eta}(a)}{\rd a} \ge - F_{d+\alpha,\eta}(0)
    \]
and therefore,
\begin{equation}\label{ineq:F}
    F_{d-1,\eta}(a) \ge F_{d-1,\eta}(0) - a F_{d+\alpha,\eta}(0).
\end{equation}

Setting $t=r^2/(2\eta)$, we can write
\begin{align}
    F_{d,\eta}(0) &=\int_0^\infty \exp\left( -\frac{1}{2\eta}r^2 \right) r^{d} \rd r
    =\int_0^\infty e^{-t} (2\eta t)^{\frac{d-1}2} \eta \rd t \nn \\
    &= 2^{\frac{d-1}2} \eta^{\frac{d+1}2} \int_0^\infty e^{-t}  t^{\frac{d-1}2} \rd t.
\end{align}
In view of the definition of the gamma function \eqref{eq:gamma}, we obtain
    \begin{equation}\label{eq:F}
        F_{d,\eta}(0) =2^{\frac{d-1}2} \eta^{\frac{d+1}2} \Gamma\left(\frac{d+1}{2}\right).
    \end{equation}
Applying the Wendel's double inequality \eqref{ineq:Wendel}
yields
\[
\frac{\Gamma\left(\frac{d+\alpha+1}{2} \right)}{\Gamma\left(\frac{d}{2} \right)} \le \left(\frac d2\right)^{\frac{\alpha+1}2}.
\]
Using \eqref{ineq:F}, \eqref{eq:F}, the above inequality and the assumption \eqref{eq:aeta}, we have
\begin{align*}
    F_{d-1,\eta}(a) &\ge F_{d-1,\eta}(0) - a F_{d+\alpha,\eta}(0) \\
    &= 2^{\frac d2-1} \eta^{\frac d2} \Gamma\left(\frac{d}{2}\right) - a 2^{\frac{d+\alpha-1}2} \eta^{\frac{d+\alpha+1}2} \Gamma\left(\frac{d+\alpha+1}{2}\right) \\
    &= 2^{\frac d2-1} \eta^{\frac d2} \Gamma\left(\frac{d}{2}\right) \left(1 - a 2^{\frac{\alpha+1}2} \eta^{\frac{\alpha+1}2} \frac{\Gamma\left(\frac{d+\alpha+1}{2} \right)}{\Gamma\left(\frac{d}{2} \right)} \right)\\
     &\ge 2^{\frac d2-1} \eta^{\frac d2} \Gamma\left(\frac{d}{2}\right) \left(1 - a (\eta d)^{\frac{\alpha+1}2} \right) \ge \frac14 (2\eta)^{\frac d2} \Gamma\left(\frac{d}{2}\right).
\end{align*}
The result \eqref{ineq:int} then follows from the above inequality and \eqref{eq:int}.
\end{proof}

We now proceed to show that the number of rejections in Algorithm \ref{alg:RGO-bundle} is bounded from above by a small constant when $\delta$ is properly chosen. In particular, as shown in Proposition \ref{prop:expected}, it only gets worse by a factor of $\exp(\delta)$ and the factor does not depend on the dimension $d$.
Hence, the implementation of RGO for $f$ is computationally efficient in practice.

\begin{proposition}\label{prop:expected}
Assume $f$ is convex and $L_\alpha$-Hölder smooth. 
If 
\begin{equation}\label{ineq:eta_mu}
    \eta \le \frac{(\alpha+1)^{\frac{2}{\alpha+1}}}{(2L_\alpha)^{\frac{2}{\alpha+1}}d}\,,
\end{equation}
then the expected number of iterations in the rejection sampling of Algorithm~\ref{alg:RGO-bundle} is at most $2\exp( \delta)$.
\end{proposition}

\begin{proof}
It is a well-known result for rejection sampling that $X \sim \pi^{X|Y}(x\mid y)$ and the probability that $X$ is accepted is
\begin{equation}\label{eq:probability}
    \mathbb{P}\left(U \leq \frac{\exp(-f_y^\eta(X))}{\exp(-h_1(X))}\right) 
=\frac{\int_{\R^d} \exp(-f_y^\eta(x)) \rd x}{\int_{\R^d} \exp(-h_1(x)) \rd x}.
\end{equation}
If follows directly from the definition of $h_2$ in \eqref{def:h} that 
\begin{align*}
    \int_{\R^d} \exp(-h_2(x)) \rd x &= \exp(-f_y^\eta(x^*)) \int_{\R^d} \exp\left(-\frac{1}{2\eta}\|x-x^*\|^2 -  \frac{L_\alpha}{\alpha+1}\|x-x^*\|^{\alpha+1}\right) \rd x 
\end{align*}
Applying Proposition \ref{prop:key} to the above yields
    \[
        \int_{\R^d} \exp(-h_2(x)) \rd x \ge \exp(-f_y^\eta(x^*)) \frac{(2\pi \eta)^{d/2}}{2}.
    \]
Note that the condition \eqref{eq:aeta} in Proposition \ref{prop:key} holds thanks to \eqref{ineq:eta_mu}.
By Lemma \ref{lem:h1h2delta}, the above inequality leads to 
\begin{equation}\label{ineq:geta}
    \int_{\R^d} \exp(-f_y^\eta(x)) \rd x \ge
\int_{\R^d} \exp(-h_2(x)) \rd x \ge \exp(-f_y^\eta(x^*)) \frac{(2\pi\eta)^{d/2}}2.
\end{equation}
Using the definition of $h_1$ in \eqref{def:tf} and Lemma \ref{lem:Gaussian}, we have
\begin{equation}\label{eq:h1}
    \int_{\R^d} \exp(-h_1(x)) \rd x 
= \exp\left( -f_y^\eta(\tx_J) + \delta \right) (2\pi \eta)^{d/2}.
\end{equation}
Using \eqref{eq:probability}, \eqref{ineq:geta} and the above identity, we conclude that
\[
    \mathbb{P}\left(U \leq \frac{\exp(-f_y^\eta(X))}{\exp(-h_1(X))}\right)
    \ge \frac12 \exp(-f_y^\eta(x^*)+f_y^\eta(\tx_J)-\delta) 
    \ge \frac12 \exp(-\delta),
\]
and the expected number of the iterations is
\[
\frac{1}{\mathbb{P}\left(U \leq \frac{\exp(-f_y^\eta(X))}{\exp(-h_1(X))}\right)}
\le 2\exp(\delta).
\]
\end{proof}

We finally bound the total complexity to sample from a log-concave distribution $\nu$ in \eqref{eq:target} with a Hölder smooth potential $f$. We combine our efficient algorithm (Algorithm \ref{alg:RGO-bundle}) of RGO for Hölder smooth potentials and the convergent results for ASF, namely Theorem \ref{Thm:LC}, to achieve this goal.

\begin{theorem}\label{thm:semi}
Assume $f$ is convex and $L_\alpha$-Hölder smooth, then Algorithm \ref{alg:ASF}, initialized with $\rho_0^X$ and stepsize $\eta \asymp 1/(L_\alpha^\frac{2}{\alpha+1} d)$, using Algorithm \ref{alg:RGO-bundle} as an RGO has the iteration-complexity bound 
    \[
        {\cal O}\left(\frac{L_\alpha^\frac{2}{\alpha+1} d W_2^2(\rho_0^X,\nu)}{\varepsilon}\right)
    \]
to achieve $\varepsilon$ error to the target $\nu\propto \exp(-f)$ in terms of KL divergence. Each RGO requires $\tilde {\cal O}\left(\frac{1}{d} \left(\frac{1}{\delta}\right)^{\frac{1-\alpha}{\alpha+1}} + 1\right)$ subgradient evaluations of $f$ and $2\exp( \delta)$ rejection steps in expectation. 
Moreover, if $\nu$ satisfies PI with constant $C_{\rm PI}>0$, 
then the iteration-complexity bound to achieve $\varepsilon$ error in terms of Chi-squared divergence is 
\[
\tilde{\cal O}\left(C_{\rm PI}L_\alpha^\frac{2}{\alpha+1} d \right).
\]
\end{theorem}

\begin{proof}
The results follow directly from Theorem \ref{Thm:LC}, Theorem \ref{thm:bundle} and Proposition \ref{prop:expected} with the choice of stepsize $\eta \asymp 1/(L_\alpha^\frac{2}{\alpha+1} d)$.
\end{proof}




\subsection{Sampling from Hybrid Potentials}\label{sec:composite}
In this subsection, we consider sampling from a log-concave distribution $\nu\propto\exp(-f(x))$ associated with a hybrid potential $f$ satisfying \eqref{ineq:composite}. This setting is a generalization of the Hölder smooth setting studied in the previous sections. It turns out that both Algorithm \ref{alg:ASF} and the implementation for RGO, Algorithm \ref{alg:RGO-bundle}, developed for Hölder smooth sampling can be applied directly to this general setting with properly chosen stepsizes.
Below, we extend the analysis in Subsection~\ref{sec:RGO} to the hybrid setting and establish corresponding complexity results.

The following lemma is a counterpart of Lemma \ref{lem:h1h2delta} in the hybrid setting. Its proof is given in Appendix~\ref{sec:proofs}.

\begin{lemma}\label{lem:composite}
Assume $f$ is convex and satisfies \eqref{ineq:composite}.
Define 
\begin{equation}\label{def:h2-hybrid}
h_2(x) := \frac{1}{2\eta} \|x-x^*\|^2 + \sum_{i=1}^n \frac{L_{\alpha_i}}{\alpha_i+1} \|x-x^*\|^{\alpha_i+1} + f_y^{\eta}(x^*).
\end{equation}
Then, $h_2(x)\ge f_y^\eta(x)$ for every $x\in \R^d$.
\end{lemma}



The next result is an analogue of the modified Gaussian integral in Proposition \ref{prop:key}.

\begin{proposition}\label{prop:key-new}
Let $\alpha_i\in[0,1]$, $a_i\ge 0$, $\eta>0$, and $d\ge 1$. If 
\begin{equation}\label{eq:aeta-new}
        \eta d \sum_{i=1}^n a_i^{\frac{2}{\alpha_i+1}} \le 1,
    \end{equation}
then
\begin{equation}\label{ineq:int-new}
    \int_{\R^d} \exp\left(-\frac{1}{2\eta}\|x\|^2 - \sum_{i=1}^n a_i\|x\|^{\alpha_i+1}\right) \rd x
\ge (2\pi \eta )^{\frac d2} \exp \left( -\frac12 + \frac{\sum_{i=1}^n (\alpha_i-1)}{4}\right).
\end{equation}
\end{proposition}

\begin{proof}
Using the Young's inequality $st\le s^p/p + t^q/q$ with 
	\[
	s = a 2^{\frac{1-\alpha}{2}} \|x\|^{\alpha+1}, \quad t= \frac{1}{2^{\frac{1-\alpha}{2}}}, \quad p= \frac{2}{\alpha+1}, \quad q= \frac{2}{1-\alpha}, 
	\]
	we obtain 
\[
a\|x\|^{\alpha+1} \le (\alpha+1) a^{\frac{2}{\alpha+1}} 2^{\frac{-2\alpha}{\alpha+1}} \|x\|^2 + \frac{1-\alpha}{4} \le a^{\frac{2}{\alpha+1}} \|x\|^2 + \frac{1-\alpha}{4},
\]
where the second inequality is due to the fact that $(\alpha+1) 2^{\frac{-2\alpha}{\alpha+1}} \le 1$ for $\alpha \in [0,1]$.
Hence, the above inequality generalizes to
\[
\sum_{i=1}^n a_i\|x\|^{\alpha_i+1} \le \sum_{i=1}^n a_i^{\frac{2}{\alpha_i+1}}  \|x\|^2 + \sum_{i=1}^n \frac{1-\alpha_i}{4}.
\]
This inequality and Lemma \ref{lem:Gaussian} imply that
\begin{align}
    &\int_{\R^d} \exp\left(-\frac{1}{2\eta}\|x\|^2 - \sum_{i=1}^n a_i\|x\|^{\alpha_i+1}\right) \rd x \nn \\
    \ge &\int_{\R^d} \exp\left(-\frac{1}{2\eta}\|x\|^2 - \sum_{i=1}^n a_i^{\frac{2}{\alpha_i+1}} \|x\|^2 - \sum_{i=1}^n \frac{1-\alpha_i}{4}\right) \rd x \nn \\
    =& \exp \left(\frac{\sum_{i=1}^n (\alpha_i-1)}{4}\right) \int_{\R^d} \exp\left(-\frac{1}{2\tilde \eta}\|x\|^2\right) \rd x \nn \\
    = & \exp \left( \frac{\sum_{i=1}^n (\alpha_i-1)}{4}\right) (2\pi \tilde \eta )^{\frac d2}  \label{ineq:int-1}
\end{align}
where
\begin{equation}\label{eq:eta'}
    \frac{1}{\tilde \eta} = \frac{1}{\eta} +  \sum_{i=1}^n a_i^{\frac{2}{\alpha_i+1}}.
\end{equation}
It follows from \eqref{eq:aeta-new} that
$\tilde \eta \ge \left(1+\frac1d \right)^{-1} \eta$.
Plugging this inequality into \eqref{ineq:int-1}, we have
\begin{align*}
    \int_{\R^d} \exp\left(-\frac{1}{2\eta}\|x\|^2 - a\|x\|^{\alpha+1}\right) \rd x 
    \ge & (2\pi \eta )^{\frac d2} \left(1+\frac1d \right)^{-\frac{d}{2}} \exp \left( \frac{\sum_{i=1}^n (\alpha_i-1)}{4}\right)\\
    \ge & (2\pi \eta )^{\frac d2} \exp \left( -\frac12 + \frac{\sum_{i=1}^n (\alpha_i-1)}{4}\right),
\end{align*}
where in the second inequality, we use the fact that
\[
\left(1+\frac1d\right)^{\frac{d}{2}} \le \exp\left(\frac12 \right).
\]
\end{proof}

With Lemma \ref{lem:composite} and Proposition \ref{prop:key-new} in hand, we can bound the complexity of Algorithm \ref{alg:RGO-bundle} as follows.
The proof is postponed to Appendix~\ref{sec:proofs}.

\begin{proposition}\label{prop:composite}
If stepsize $\eta$ satisfies
\begin{equation}\label{ineq:assumption3}
    \eta d \sum_{i=1}^n \left(\frac{L_{\alpha_i}}{\alpha_i+1}\right)^{\frac{2}{\alpha_i+1}} \le 1,
\end{equation}
then rejection steps in Algorithm \ref{alg:RGO-bundle} take at most $\exp\left(\delta +\frac12 + \frac{\sum_{i=1}^n (1-\alpha_i)}{4}\right)$ iterations in expectation.
\end{proposition}

Through the above arguments, we show that Algorithm \ref{alg:RGO-bundle} designed for Hölder smooth potentials is equally effective for hybrid potentials satisfying \eqref{ineq:composite}. Combining Proposition~\ref{prop:composite} and Theorem~\ref{thm:bundle-hybrid} with the convergence results for ASF, we obtain the following iteration-complexity bounds for sampling from hybrid potentials. The proof is similar to that of Theorem \ref{thm:semi} and is thus omitted.

\begin{theorem}\label{thm:all}
Assume $f$ is a convex and satisfies \eqref{ineq:composite}. Consider Algorithm \ref{alg:ASF}, initialized with $\rho_0^X$ and stepsize $\eta$ satisfies \eqref{ineq:assumption3}, using Algorithm \ref{alg:RGO-bundle} as a RGO. Each RGO requires $\tilde {\cal O}\left(\frac{1}{d\delta} + 1\right)$ subgradient evaluations of $f$ and in expectation
$\exp\left(\delta +\frac12 + \frac{\sum_{i=1}^n (1-\alpha_i)}{4}\right)$ rejection steps.
The total complexity of Algorithm \ref{alg:ASF} to achieve $\varepsilon$ error in terms of KL divergence is 
    \[
        {\cal O}\left(\frac{\sum_{i=1}^n \left(\frac{L_{\alpha_i}}{\alpha_i+1}\right)^{\frac{2}{\alpha_i+1}} d W_2^2(\rho_0^X,\nu)}{\varepsilon}\right).
    \]
Moreover, if $\nu$ satisfies PI with constant $C_{\rm PI}>0$, then total complexity to achieve $\varepsilon$ error in terms of Chi-squared divergence is 
    \[
        \tilde{\cal O}\left(\sum_{i=1}^n \left(\frac{L_{\alpha_i}}{\alpha_i+1}\right)^{\frac{2}{\alpha_i+1}} d C_{\rm PI}\right).
    \]
\end{theorem}




\section{Conclusions}\label{sec:conclusion}

In this paper, we study proximal algorithms for both optimization and sampling lacking smoothness.
We first establish the complexity bounds of the regularized cutting-plane method for solving proximal subproblem \eqref{eq:prox-sub}, where $f$ is convex and satisfies either \eqref{ineq:semi-smooth} (Hölder smooth) or \eqref{ineq:composite} (hybrid). This efficient implementation gives an approximate solution to the proximal map in optimization, which is the core of both proximal optimization and sampling algorithms.

For optimization, we develop APBM using a novel adaptive stepsize strategy in the proximal point method and the approximate proximal map to solve each proximal subproblem. The proposed APBM is a universal method as it does not requuire any problem-dependent parameters as input.

For sampling, we propose an efficient method based on rejection sampling and the approximate proximal map to realize the RGO, which is a proximal sampling oracle. Finally, combining the sampling complexity of RGO and the complexity bounds of ASF, which is a counterpart of the proximal point method in sampling, we establish the complexity bounds of the proximal sampling algorithm in both Hölder smooth and hybrid settings.

This paper provides a unified perspective to study proximal optimization and sampling algorithms, while many other interesting questions remain open.
First, APBM is only optimal when $\alpha=0$, i.e., $f$ is Lipschitz continuous. We are interested in developing a universal method that is optimal for any $\alpha\in [0,1]$. 
One possible direction is to incorporate the acceleration technique into both the regularized cutting-plane method and the PPF.
Second, as acceleration methods are widely used in optimization to obtain optimal performance, accelerated proximal sampling algorithms are less explored. It is worth investigating a counterpart of the accelerated proximal point method \cite{MonteiroSvaiterAcceleration} in sampling.
Finally, we develop APBM as a universal method for non-smooth optimization, and it would be equally important to design a universal method for sampling.


\bibliographystyle{plain}
\bibliography{ref.bib}

\appendix








\section{Technical results}\label{sec:technical}

This section collects technical results that are useful in the paper.

\begin{lemma}[Gaussian integral]\label{lem:Gaussian}
For any $\eta>0$,
    \[
    \int_{\R^d} \exp\left(-\frac{1}{2\eta}\|x\|^2\right) \rd x = (2\pi \eta)^{d/2}.
    \]
\end{lemma}

The following lemma provides both lower and upper bounds on the ratio of gamma functions. Its proof can be found in \cite{wendel1948note}.

\begin{lemma}[Wendel's double inequality]
For $0<s<1$ and $t>0$, the gamma function defined as in \eqref{eq:gamma}
satisfies 
\[
\left(\frac{t}{t+s}\right)^{1-s} \le \frac{\Gamma(t+s)}{t^s \Gamma(t)} \le 1,
\]
or equivalently,
\begin{equation}\label{ineq:Wendel}
    t^{1-s} \le \frac{\Gamma(t+1)}{\Gamma(t+s)} \le (t+s)^{1-s}.
\end{equation}
\end{lemma}

\begin{lemma}\label{lem:tech}
Assume $f$ is convex and $L_\alpha$-semi-smooth (i.e., satisfying \eqref{ineq:semi-smooth}, then \eqref{ineq:semi} holds for every $u,v \in \R^d$.
Assume $f$ is convex and satisfies \eqref{ineq:composite}, then \eqref{ineq:hybrid} holds for every $u,v \in \R^d$.
\end{lemma}
\begin{proof}
We first consider the case when $f$ is convex and $L_\alpha$-semi-smooth.
It is easy to see that
		\begin{align*}
			f(u)&=f(v)+\int_{0}^{1} \inner{f'(v+\tau (v-u))} {u-v} \rd\tau \\
			&=f(v) + \inner{f'(v)}{u-v}+\int_{0}^{1} \inner{f'(v+\tau (v-u)) - f'(v)} {u-v} \rd\tau.
		\end{align*}
Using the above identity, the Cauchy-Schwarz inequality, and \eqref{ineq:semi-smooth}, we have
		\begin{align*}
			f(u) - f(v) - \inner{f'(v)}{u-v} 
			&= \int_{0}^{1}  \inner{f'(v+\tau (v-u)) - f'(v)} {u-v} \rd\tau \\
			&\le \int_{0}^{1} \left\| f'(v+\tau (v-u)) - f'(v)\right\| \|u-v\|  \rd\tau \\
			&\le \int_{0}^{1}  L_\alpha \tau^\alpha \|u-v\|^{\alpha+1}  \rd\tau
		= \frac{L_\alpha}{\alpha+1}\|u-v\|^{\alpha+1}.
		\end{align*}
Hence, \eqref{ineq:semi} holds. More generally, if $f$ satisfies \eqref{ineq:composite}, then  \eqref{ineq:hybrid} follows the same argument.
\end{proof}

\begin{lemma}\label{lem:Lip}
    Consider $\phi(t) = |t|^p$ and $\phi'(t) = p\,\operatorname{sign}(t)\,|t|^{p-1}$ for $t\in \R$ and some $p\in [1,2]$. Then, for any $u,v \in \R$, we have 
    \[
    |\phi'(u)-\phi'(v)| \le p\,2^{2-p}\,|u-v|^{p-1}.
    \]
\end{lemma}

\begin{proof}
Let $r := p-1 \in [0,1]$.
We consider the following two cases and prove
\[
|\phi'(u)-\phi'(v)| \le p\,2^{1-r}\,|u-v|^{r}.
\]

\medskip
\noindent\textbf{Case 1: $u v \ge 0$.}
Here, $\operatorname{sign}(u)=\operatorname{sign}(v)$, so
\[
|\phi'(u)-\phi'(v)|
= p\,\big||u|^{r}-|v|^{r}\big|.
\]
Without loss of generality, assume $a=|u|\ge b=|v|\ge 0$.
By the subadditivity of the function $x \mapsto x^{r}$ with $r\in [0,1]$, we have $(a-b)^{r} \ge a^{r} - b^{r}$ for all $a\ge b\ge 0$.
Hence
\[
|\,|u|^{r}-|v|^{r}\,|
= a^{r}-b^{r}
\le (a-b)^{r}
= \big||u|-|v|\big|^{r}
= |u-v|^{r}.
\]
Therefore,
\[
|\phi'(u)-\phi'(v)| \le p\,|u-v|^{r}.
\]

\medskip
\noindent\textbf{Case 2: $u v < 0$.}
In this case, $\operatorname{sign}(u)=-\operatorname{sign}(v)$, so
\[
|\phi'(u)-\phi'(v)| = p\,(|u|^{r}+|v|^{r}).
\]
By the concavity of $x\mapsto x^{r}$, we have
\[
|u|^{r} + |v|^{r} \le 2^{1-r} (|u|+|v|)^{r}
= 2^{1-r} |u - v|^{r},
\]
and thus
\[
|\phi'(u)-\phi'(v)| \le p\,2^{1-r}\,|u-v|^{r}.
\]
The conclusion immediately follows from the above two cases.
\end{proof}

\section{Missing proofs in Subsection \ref{sec:composite}}\label{sec:proofs}

\noindent
{\bf Proof of Lemma \ref{lem:composite}:}
It follows from the same argument as in the proof of Lemma \ref{lem:h1h2delta} that \eqref{eq:equal} holds. Using \eqref{eq:equal}, Lemma \ref{lem:bundle}(e), and \eqref{ineq:hybrid} with $(u,v)=(x,x^*)$, we conclude that 
\[
f_y^\eta(x) - f_y^\eta(x^*) \le \sum_{i=1}^n \frac{L_{\alpha_i}}{\alpha_i+1} \|x-x^*\|^{\alpha_i+1} + \frac{1}{2\eta}\|x-x^*\|^2.
\]
The lemma immediately follows from the above inequality and the definition of $h_2$ in \eqref{def:h2-hybrid}.
\QEDA

\vgap

\noindent
{\bf Proof of Proposition \ref{prop:composite}:}
If follows directly from the definition of $h_2$ in \eqref{def:h2-hybrid} that 
\begin{align*}
    \int_{\R^d} \exp(-h_2(x)) \rd x &= \exp(-f_y^\eta(x^*)) \int_{\R^d} \exp\left(-\frac{1}{2\eta} \|x-x^*\|^2 - \sum_{i=1}^n \frac{L_{\alpha_i}}{\alpha_i+1} \|x-x^*\|^{\alpha_i+1}\right) \rd x.
\end{align*}
It is easy to see that \eqref{ineq:assumption3} implies that \eqref{eq:aeta-new} holds with $a_i=\frac{L_{\alpha_i}}{\alpha_i+1}$. Hence, by Proposition \ref{prop:key-new}, we have \eqref{prop:key-new} holds with $a_i=\frac{L_{\alpha_i}}{\alpha_i+1}$, i.e.,
\[
\int_{\R^d} \exp\left(-\frac{1}{2\eta} \|x-x^*\|^2 - \sum_{i=1}^n \frac{L_{\alpha_i}}{\alpha_i+1} \|x-x^*\|^{\alpha_i+1}\right) \rd x \ge (2\pi \eta )^{\frac d2} \exp \left( -\frac12 + \frac{\sum_{i=1}^n (\alpha_i-1)}{4}\right).
\]
The above two inequalities and Lemma \ref{lem:composite} imply that
\[
\int_{\R^d} \exp(-f_y^\eta(x)) \rd x \ge
\int_{\R^d} \exp(-h_2(x)) \rd x \ge (2\pi \eta )^{\frac d2} \exp \left( -f_y^\eta(x^*)-\frac12 + \frac{\sum_{i=1}^n (\alpha_i-1)}{4}\right).
\]
As in the proof of Proposition \ref{prop:expected}, \eqref{eq:probability} and \eqref{eq:h1} hold.
Using \eqref{eq:probability}, \eqref{eq:h1}, and the above inequality, we have
\[
\mathbb{P}\left(U \leq \frac{\exp(-f_y^\eta(X))}{\exp(-h_1(X))}\right) 
\ge \exp\left(f_y^\eta(\tx_J) - f_y^\eta(x^*) - \delta -\frac12 + \frac{\sum_{i=1}^n (\alpha_i-1)}{4}\right).
\]
 The above inequality and the fact that $f_y^\eta(\tx_J) \ge f_y^\eta(x^*)$ immediately imply that
\[
\frac{1}{\mathbb{P}\left(U \leq \frac{\exp(-f_y^\eta(X))}{\exp(-h_1(X))}\right) }
\le \exp\left(\delta +\frac12 + \frac{\sum_{i=1}^n (1-\alpha_i)}{4}\right).
\]
\QEDA

\end{document}